\documentclass{amsart}

\usepackage[letterpaper,margin=1in]{geometry}
\usepackage{amssymb,amsmath,amsfonts,graphicx,amsthm,mathrsfs,amsbsy}
\usepackage{dsfont}
\usepackage{color}

\usepackage{bmpsize}
\usepackage{algorithm}
\usepackage{algorithmic}

\usepackage{amscd,verbatim}
\usepackage{enumerate}
\usepackage{float}
\usepackage[caption = false]{subfig}
\usepackage{bm}

\DeclareMathOperator*{\E}{\mathbb{E}}
\DeclareMathOperator*{\bP}{\mathbb{P}}

\newtheorem{theorem}{Theorem}[section]

\newtheorem{lemma}[theorem]{Lemma}

\newtheorem{proposition}[theorem]{Proposition}

\newtheorem{corollary}[theorem]{Corollary}
\newtheorem{example}[theorem]{Example}

\DeclareGraphicsExtensions{.png,.jpg,.pdf,.eps}

\title{Uniform Error Estimates for the Lanczos Method}
\keywords{symmetric eigenvalue problem, Krylov subspace method, Lanczos method}
\subjclass{15A18, 65F15, 65F50}

\author{John C. Urschel} 
\address{Department of Mathematics, Massachusetts Institute of Technology, Cambridge, MA, USA.}
\email{urschel@mit.edu}

\everymath{\displaystyle}

\begin{document}

\begin{abstract}
 The Lanczos method is one of the most powerful and fundamental techniques for solving an extremal symmetric eigenvalue problem. Convergence-based error estimates depend heavily on the eigenvalue gap. In practice, this gap is often relatively small, resulting in significant overestimates of error. One way to avoid this issue is through the use of uniform error estimates, namely, bounds that depend only on the dimension of the matrix and the number of iterations.  In this work, we prove explicit upper and lower uniform error estimates for the Lanczos method. These lower bounds, paired with numerical results, imply that the maximum error of $m$ iterations of the Lanczos method over all $n \times n$ symmetric matrices does indeed depend on the dimension $n$ in practice. The improved bounds for extremal eigenvalues translate immediately to error estimates for the condition number of a symmetric positive definite matrix. In addition, we prove more specific results for matrices that possess some level of eigenvalue regularity or whose eigenvalues converge to some limiting empirical spectral distribution. Through numerical experiments, we show that the theoretical estimates of this paper do apply to practical computations for reasonably sized matrices.
\end{abstract}

\maketitle

\section{Introduction}

The computation of extremal eigenvalues of matrices is one of the most fundamental problems in numerical linear algebra. When a matrix is large and sparse, methods such as the Jacobi eigenvalue algorithm and QR algorithm become computationally infeasible, and, therefore, techniques that take advantage of the sparsity of the matrix are required. Krylov subspace methods are a powerful class of techniques for approximating extremal eigenvalues, most notably the Arnoldi iteration for non-symmetric matrices and the Lanczos method for symmetric matrices. 

The Lanczos method is a technique that, given a symmetric matrix $A \in \mathbb{R}^{n \times n}$ and an initial vector $b \in \mathbb{R}^n$, iteratively computes a tridiagonal matrix $T_m \in \mathbb{R}^{m \times m}$ that satisfies $T_m = Q^T_m A Q_m$, where $Q_m \in \mathbb{R}^{n \times m}$ is an orthonormal basis of the Krylov subspace
$$ \mathcal{K}_m(A,b) = \text{span}\{b, Ab, ..., A^{m-1} b\} .$$
The eigenvalues of $T_m$, denoted by $\lambda^{(m)}_1(A,b) \ge ... \ge \lambda^{(m)}_m(A,b)$, are the Rayleigh-Ritz approximations to the eigenvalues $\lambda_1(A) \ge ... \ge \lambda_n(A)$ of $A$ on $\mathcal{K}_m(A,b)$, and, therefore, are given by
\begin{eqnarray}\label{eqn:minmax}
\lambda_i^{(m)}(A,b) &=& \min_{\substack{U \subset \mathcal{K}_m(A,b)\\ \text{dim}(U) = m+1-i}} \, \max_{\substack{x \in U \\ x \ne 0}} \, \frac{x^T A x}{x^T x}, \qquad \qquad i = 1,...,m,
\end{eqnarray} 
or, equivalently,
\begin{eqnarray}\label{eqn:maxmin}
\lambda_i^{(m)} (A,b)  &=& \max_{\substack{U \subset \mathcal{K}_m(A,b)\\ \text{dim}(U) = i}} \, \min_{\substack{x \in U \\ x \ne 0}} \, \frac{x^T A x}{x^T x}, \qquad \qquad i = 1,...,m.
\end{eqnarray}

This description of the Lanczos method is sufficient for a theoretical analysis of error (i.e., without round-off error), but, for completeness, we provide a short description of the Lanczos method (when $\mathcal{K}_m(A,b)$ is full-rank) in Algorithm \ref{ag:lan} \cite{trefethen1997numerical}. For a more detailed discussion of the nuances of practical implementation and techniques to minimize the effects of round-off error, we refer the reader to \cite[Section 10.3]{golub2012matrix}. If $A$ has $\nu n$ non-zero entries, then Algorithm \ref{ag:lan} outputs $T_m$ after approximately $(2 \nu +8) m n$ floating point operations. A number of different techniques, such as the divide-and-conquer algorithm with the fast multipole method \cite{coakley2013fast}, can easily compute the eigenvalues of a tridiagonal matrix. The complexity of this computation is negligible compared to the Lanczos algorithm, as, in practice, $m$ is typically orders of magnitude less than $n$.

Equation (\ref{eqn:minmax}) for the Ritz values $\lambda_i^{(m)}$ illustrates the significant improvement that the Lanczos method provides over the power method for extremal eigenvalues. Whereas the power method uses only the iterate $A^{m} b$ as an approximation of the largest magnitude eigenvalue, the Lanczos method uses the span of all of the iterates of the power method (given by $\mathcal{K}_{m+1}(A,b)$). However, the analysis of the Lanczos method is significantly more complicated than that of the power method. Error estimates for extremal eigenvalue approximation using the Lanczos method have been well studied, most notably by Kaniel \cite{kaniel1966estimates}, Paige \cite{paige1971computation}, Saad \cite{saad1980rates}, and Kuczynski and Wozniakowski \cite{kuczynski1992estimating} (other notable work includes \cite{del1997randomized,drineas2018structural,kuczynski1994probabilistic,musco2015randomized,parlett1982estimating,simchowitz2017gap,simchowitz2018tight,van2001computing}). The work of Kaniel, Paige, and Saad focused on the convergence of the Lanczos method as $m$ increases, and, therefore, their results have strong dependence on the spectrum of the matrix $A$ and the choice of initial vector $b$.
For example, a standard result of this type is the estimate
\begin{equation} \label{eqn:kp}
\frac{\lambda_1(A) - \lambda_1^{(m)}(A,b)}{\lambda_1(A) - \lambda_n(A)} \le \left( \frac{\tan \angle(b, \varphi_1)}{T_{m-1}(1 + 2 \gamma)} \right)^2, \qquad   \quad \gamma = \frac{\lambda_1 (A) - \lambda_2(A)}{\lambda_2(A) - \lambda_n(A)},
\end{equation}
where $\varphi_1$ is the eigenvector corresponding to $\lambda_1$, and $T_{m-1}$ is the $(m-1)^{th}$ degree Chebyshev polynomial of the first kind \cite[Theorem 6.4]{saad2011numerical}.
Kuczynski and Wozniakowski took a quite different approach, and estimated the maximum expected relative error $(\lambda_1 - \lambda_1^{(m)})/\lambda_1$ over all $n \times n$ symmetric positive definite matrices, resulting in error estimates that depend only on the dimension $n$ and the number of iterations $m$. They produced the estimate
\begin{equation}\label{eqn:kw}
\sup_{A \in \mathcal{S}_{++}^n} \E_{b \sim \mathcal{U}(S^{n-1})} \left[ \frac{\lambda_1(A) - \lambda_1^{(m)}(A,b)}{\lambda_1(A)} \right] \le .103 \, \frac{ \ln^2 (n (m-1)^4)}{(m-1)^2},
\end{equation}
for all $n \ge 8$ and $m \ge 4$, where $\mathcal{S}_{++}^n$ is the set of all $n \times n$ symmetric positive definite matrices, and the expectation is with respect to the uniform probability measure on the hypersphere $S^{n-1}= \{ x \in \mathbb{R}^n \, | \, \| x \| = 1 \}$ \cite[Theorem 3.2]{kuczynski1992estimating}. One can quickly verify that equation (\ref{eqn:kw}) also holds when the $\lambda_1(A)$ term in the denominator is replaced by $\lambda_1(A)-\lambda_n(A)$, and the supremum over $\mathcal{S}_{++}^n$ is replaced by the maximum over the set of all $n\times n$ symmetric matrices, denoted by $\mathcal{S}^n$.

Both of these approaches have benefits and drawbacks. If an extremal eigenvalue is known to have a reasonably large eigenvalue gap (based on application or construction), then a distribution dependent estimate provides a very good approximation of error, even for small $m$. However, if the eigenvalue gap is not especially large, then distribution dependent estimates can significantly overestimate error, and estimates that depend only on $n$ and $m$ are preferable. This is illustrated by the following elementary, yet enlightening, example.

\begin{algorithm}[!tb]
	\caption{Lanczos Method \label{ag:lan}}
	\begin{enumerate}
		\item[] { \bf Input:} symmetric matrix $A \in \mathbb{R}^{n \times n}$, vector $b \in \mathbb{R}^n$, number of iterations $m$.
		\item[] { \bf Output:} symmetric tridiagonal matrix $T_m \in \mathbb{R}^{m \times m}$, $T_m(i,i) = \alpha_i$, 
		\item[] $ \qquad \qquad \,$ $T_m(i,i+1) = \beta_i$, satisfying $T_m = Q_m^T A Q_m$, where $Q_m = [q_1 \; ... \; q_m]$.
		\item[]
		\item[] { \bf Set} $\beta_0=0$, $q_0 = 0$, $q_1 = b/\| b \|$
		\item[] {\bf For} $i = 1,...,m$ 
		\begin{enumerate}
			\item[] $v = A q_i$
			\item[] $\alpha_i = q_i^T v$
			\item[] $v = v - \alpha_i q_i -\beta_{i-1} q_{i-1} $
			\item[] $\beta_i = \|v\|$
			\item[] $q_{i+1} = v/\beta_i$
		\end{enumerate}
		
	\end{enumerate}
\end{algorithm}

\begin{example}\label{ex:tridiag}
	Let $A \in S_{++}^n$ be the tridiagonal matrix resulting from the discretization of the Laplacian operator on an interval with Dirichlet boundary conditions, namely, $A_{i,i} = 2$ for $i=1,...,n$ and $A_{i,i+1} = A_{i+1,i} = -1$ for $i =1,...,n-1$. The eigenvalues of $A$ are given by $\lambda_i(A) = 2 + 2 \cos \left( i \pi /(n+1) \right)$, $i=1,...,n$. Consider the approximation of $\lambda_1(A)$ by $m$ iterations of the Lanczos method. For a random choice of $b$, the expected value of $\tan^2 \angle(b, \varphi_1)$ is $(1+o(1)) n$. If $\tan^2 \angle(b, \varphi_1) = C n$ for some constant $C$, then (\ref{eqn:kp}) produces the estimate
	\begin{align*}
	\frac{\lambda_1(A) - \lambda_1^{(m)}(A,b)}{\lambda_1(A) - \lambda_n(A)} &\le C n \, T_{m-1} \left(1 + 2 \tan \left( \frac{\pi}{2(n+1)}\right) \tan \left( \frac{ 3\pi}{2(n+1)} \right) \right)^{-2} \\
	&\simeq n (1 + O(n^{-1}) )^{-m} \simeq n.
	\end{align*}
	In this instance, the estimate is a trivial one for all choices of $m$, which varies greatly from the error estimate (\ref{eqn:kw}) of order $\ln^2 n / m^2$. The exact same estimate holds for the smallest eigenvalue $\lambda_n(A)$ when the corresponding bounds are applied, since $4I - A$ is similar to $A$.
	
	Now, consider the approximation of the largest eigenvalue of $B = A^{-1}$ by $m$ iterations of the Lanczos method. The matrix $B$ possesses a large gap between the largest and second-largest eigenvalue, which results in a value of $\gamma$ for (\ref{eqn:kp}) that remains bounded below by a constant independent of $n$, namely 
	$$\gamma = \left(2 \cos \left( \pi/(n+1) \right) + 1 \right) / \left(2 \cos \left( \pi/(n+1) \right) - 1 \right).$$
	Therefore, in this instance, the estimate (\ref{eqn:kp}) illustrates a constant convergence rate, produces non-trivial bounds (for typical $b$) for $m = \Theta(\ln n)$, and is preferable to the error estimate (\ref{eqn:kw}) of order $\ln^2 n / m^2$.
\end{example}

More generally, if a matrix $A$ has eigenvalue gap $\gamma \lesssim n^{-\alpha}$ and the initial vector $b$ satisfies $\tan^2 \angle (b, \varphi_1) \gtrsim n$, then the error estimate (\ref{eqn:kp}) is a trivial one for $m \lesssim n^{\alpha/2}$. This implies that the estimate (\ref{eqn:kp}) is most useful when the eigenvalue gap is constant or tends to zero very slowly. When the gap is not especially large (say, $n^{-\alpha}$, $\alpha$ constant), then uniform error estimates are preferable for small values of $m$. In this work, we focus on uniform bounds, namely, error estimates that hold uniformly over some large set of matrices (typically $\mathcal{S}^n_{++}$ or $\mathcal{S}^n$). We begin by recalling some of the key existing uniform error estimates for the Lanczos method.

\subsection{Related Work}
Uniform error estimates for the Lanczos method have been produced almost exclusively for symmetric positive definite matrices, as error estimates for extremal eigenvalues of symmetric matrices can be produced from estimates for $S_{++}^n$ relatively easily. The majority of results apply only to either $\lambda_1(A)$, $\lambda_n(A)$, or some function of the two (i.e., condition number). All estimates are probabilistic and take the initial vector $b$ to be uniformly distributed on the hypersphere. Here we provide a brief description of some key uniform error estimates previously produced for the Lanczos method.

In \cite{kuczynski1992estimating}, Kuczynski and Wozniakowski produced a complete analysis of the power method and provided a number of upper bounds for the Lanczos method. Most notably, they produced error estimate (\ref{eqn:kw}) and provided the following upper bound for the probability that the relative error $(\lambda_1 - \lambda_1^{(m)})/\lambda_1$ is greater than some value $\epsilon$:
\begin{equation} \label{eqn:kwprob}
\sup_{A \in \mathcal{S}_{++}^n} \bP_{b \sim \mathcal{U}(S^{n-1})} \left[ \frac{\lambda_1(A) - \lambda^{(m)}_1(A,b)}{\lambda_1(A)} > \epsilon   \right]  \le 1.648 \sqrt{n} e^{-\sqrt{\epsilon}(2m-1)}.
\end{equation}
However, the authors were unable to produce any lower bounds for (\ref{eqn:kw}) or (\ref{eqn:kwprob}), and stated that a more sophisticated analysis would most likely be required. In the same paper, they performed numerical experiments for the Lanczos method that produced errors of the order $m^{-2}$, which led the authors to suggest that the error estimate (\ref{eqn:kw}) may be an overestimate, and that the $\ln^2 n$ term may be unnecessary. 

In \cite{kuczynski1994probabilistic}, Kuczynski and Wozniakowski noted that the above estimates immediately translate to relative error estimates for minimal eigenvalues when the error $\lambda^{(m)}_m - \lambda_n$ is considered relative to $\lambda_1$ or $\lambda_1-\lambda_n$ (both normalizations can be shown to produce the same bound). However, we can quickly verify that there exist sequences in $\mathcal{S}^{n}_{++}$ for which the quantity $\E\big[( \lambda_m^{(m)} - \lambda_n)/\lambda_n \big]$ is unbounded. These results for minimal eigenvalues, combined with (\ref{eqn:kw}), led to error bounds for estimating the condition number of a matrix. Unfortunately, error estimates for the condition number face the same issue as the quantity $( \lambda_m^{(m)} - \lambda_n)/\lambda_n$, and therefore, only estimates that depend on the value of the condition number can be produced.

The proof technique used to produce (\ref{eqn:kw}) works specifically for the quantity $\E\big[(\lambda_1 - \lambda_1^{(m)})/\lambda_1 \big]$ (i.e., the $1$-norm), and does not carry over to more general $p$-norms of the form $\E\big[ ((\lambda_1 - \lambda_1^{(m)})/\lambda_1)^p \big]^{1/p}$, $p \in (1,\infty)$. Later, in \cite[Theorem 5.2, $r=1$]{del1997randomized}, Del Corso and Manzini produced an upper bound of the order $(\sqrt{n}/m)^{1/p}$ for arbitrary $p$-norms, given by
\begin{equation} \label{eqn:upperp}
\sup_{A \in \mathcal{S}_{++}^n} \E_{b \sim \mathcal{U}(S^{n-1})} \left[ \left( \frac{\lambda_1(A) - \lambda_1^{(m)}(A,b)}{\lambda_1(A)}\right)^p \right]^{\frac{1}{p}} \lesssim \frac{1}{m^{1/p}} \left(\frac{\Gamma\left(p-\tfrac{1}{2}\right)\Gamma\left(\tfrac{n}{2}\right)}{\Gamma(p)\Gamma\left(\tfrac{n-1}{2}\right)} \right)^{\frac{1}{p}} . 
\end{equation}
This bound is clearly worse than (\ref{eqn:kw}), and better bounds can be produced for arbitrary $p$ simply by making use of (\ref{eqn:kwprob}). Again, the authors were unable to produce any lower bounds.

More recently, the machine learning and optimization community has become increasingly interested in the problem of approximating the top singular vectors of a symmetric matrix by randomized techniques (for a review, see \cite{halko2011finding}). This has led to a wealth of new results in recent years regarding classical techniques, including the Lanczos method. In \cite{musco2015randomized}, Musco and Musco considered a block Krylov subspace algorithm similar to the block Lanczos method and showed that with high probability their algorithm achieves an error of $\epsilon$ in $m = O(\epsilon^{-1/2} \ln n)$ iterations, matching the bound \ref{eqn:kwprob} for a block size of one. Very recently, a number of lower bounds for a wide class of randomized algorithms were shown in \cite{simchowitz2017gap,simchowitz2018tight}, all of which can be applied to the Lanczos method as a corollary. First, the authors showed that the dependence on $n$ in the above upper bounds was in fact necessary. In particular, it follows from \cite[Theorem A.1]{simchowitz2018tight} that $O(\log n)$ iterations are necessary to obtain a non-trivial error. In addition, as a corollary of their main theorem \cite[Theorem 1]{simchowitz2018tight}, there exist $c_1,c_2 >0$ such that for all $\epsilon \in (0,1)$ there exists an $n_o = \text{poly}(\epsilon^{-1})$ such that for all $n \ge n_o$ there exists a random matrix $A \in \mathcal{S}^n$ such that
\begin{equation}\label{eqn:lwrbnd}
\mathbb{P} \left[ \frac{\rho(A) - \rho(T)}{\rho(A)} \ge \frac{\epsilon}{12} \right] \ge 1 - e^{-n^{c_2}} \qquad \text{ for all }  m \le c_1 \epsilon^{-1/2} \ln n,
\end{equation}
where $\rho (\cdot)$ is the spectral radius of a matrix.

\subsection{Contributions and Remainder of Paper}
In what follows, we prove improved upper bounds for the maximum expected error of the Lanczos method in the $p$-norm, $p \ge 1$, and combine this with nearly-matching asymptotic lower bounds with explicit constants. These estimates can be found in Theorem \ref{thm:main}. The upper bounds result from using a slightly different proof technique than that of $(\ref{eqn:kw})$, which is more robust to arbitrary $p$-norms. Comparing the lower bounds of Theorem \ref{thm:main} to the estimate (\ref{eqn:lwrbnd}), we make a number of observations. Whereas (\ref{eqn:lwrbnd}) results from a statistical analysis of random matrices, our estimates follow from taking an infinite sequence of non-random matrices with explicit eigenvalues and using the theory of orthogonal polynomials. Our estimate for $m = O(\ln n)$ is slightly worse than (\ref{eqn:lwrbnd}) by a factor of $\ln^2 \ln n$, but makes up for this in the form of an explicit constant. The estimate for $m = o(n^{1/2} \ln^{-1/2} n)$ does not have $n$ dependence, but illustrates a useful lower bound, as it has an explicit constant and the $\ln n$ term becomes negligible as $m$ increases. The results (\ref{eqn:lwrbnd}) do not fully apply to this regime, as the polynomial in this estimate is at least cubic (see \cite[Theorem 6.1]{simchowitz2018tight} for details).

To complement these bounds, we also provide an error analysis for matrices that have a certain structure. In Theorem \ref{thm:eigclust}, we produce improved dimension-free upper bounds for matrices that have some level of eigenvalue ``regularity" near $\lambda_1$. In addition, in Theorem \ref{thm:conv}, we prove a powerful result that can be used to determine, for any fixed $m$, the asymptotic relative error for any sequence of matrices $X_n \in \mathcal{S}^{n}$, $n = 1,2,...$, that exhibits suitable convergence of its empirical spectral distribution. Later, we perform numerical experiments that illustrate the practical usefulness of this result. Theorem \ref{thm:conv}, combined with estimates for Jacobi polynomials (see Proposition \ref{prop:jac2}), illustrates that the inverse quadratic dependence on the number of iterations $m$ in the estimates produced throughout this paper does not simply illustrate the worst case, but is actually indicative of the typical case in some sense.

In Corollary \ref{cor:eig_lb} and Theorem \ref{thm:arb_eig}, we produce results similar to Theorem \ref{thm:main} for arbitrary eigenvalues $\lambda_i$. The lower bounds follow relatively quickly from the estimates for $\lambda_1$, but the upper bounds require some mild assumptions on the eigenvalue gaps of the matrix. These results mark the first uniform-type bounds for estimating arbitrary eigenvalues by the Lanczos method. In addition, in Corollary \ref{cor:cond_exp}, we translate our error estimates for the extremal eigenvalues $\lambda_1(A)$ and $\lambda_n(A)$ into error bounds for the condition number of a symmetric positive definite matrix, but, as previously mentioned, the relative error of the condition number of a matrix requires estimates that depends on the condition number itself. Finally, we present numerical experiments that support the accuracy and practical usefulness of the theoretical estimates detailed above.

The remainder of the paper is as follows. In Section 2, we prove basic results regarding relative error and make a number of fairly straightforward observations. In Section 3, we prove asymptotic lower bounds and improved upper bounds for the relative error in an arbitrary $p$-norm. In Section 4, we produce a dimension-free error estimate for a large class of matrices and prove a theorem that can be used to determine the asymptotic relative error for any fixed $m$ and sequence of matrices $X_n \in \mathcal{S}^{n}$, $n = 1,2,...,$ with suitable convergence of its empirical spectral distribution. In Section 5, under some mild additional assumptions, we prove a version of Theorem \ref{thm:main} for arbitrary eigenvalues $\lambda_i$, and extend our results for $\lambda_1$ and $\lambda_n$ to the condition number of a symmetric positive definite matrix. Finally, in Section 6, we perform a number of experiments and discuss how the numerical results compare to the theoretical estimates in this work.

\section{Preliminary Results}
Because the Lanczos method applies only to symmetric matrices, all matrices in this paper are assumed to belong to $\mathcal{S}^n$. The Lanczos method and the quantity $(\lambda_1(A) - \lambda_1^{(m)}(A,b))/(\lambda_1(A) -\lambda_n(A))$ are both unaffected by shifting and scaling, and so any maximum over $A \in \mathcal{S}^n$ can be replaced by a maximum over all $A \in \mathcal{S}^n$ with $\lambda_1(A)$ and $\lambda_n(A)$ fixed. Often for producing upper bounds, it is convenient to choose $\lambda_1(A) =1$ and $\lambda_n(A) = 0$. For the sake of brevity, we will often write $\lambda_i(A)$ and $\lambda_j^{(m)}(A,b)$ as $\lambda_i$ and $\lambda_j^{(m)}$ when the associated matrix $A$ and vector $b$ are clear from context.

We begin by rewriting expression $(\ref{eqn:maxmin})$ for $\lambda_1^{(m)}$ in terms of polynomials. The Krylov subspace $\mathcal{K}_m(A,b)$ can be alternatively defined as
$$ \mathcal{K}_m(A,b) = \{ P(A) b \, | \, P \in \mathcal{P}_{m-1} \},$$
where $\mathcal{P}_{m-1}$ is the set of all real-valued polynomials of degree at most $m-1$. Suppose $A \in \mathcal{S}^n$ has eigendecomposition $A = Q \Lambda Q^T$, where $Q \in \mathbb{R}^{n\times n}$ is an orthogonal matrix and $\Lambda \in \mathbb{R}^{n \times n}$ is the diagonal matrix satisfying $\Lambda(i,i) = \lambda_i(A)$, $i = 1,...,n$. Then we have the relation
$$ \lambda_1^{(m)}(A,b) = \max_{\substack{x \in \mathcal{K}_{m}(A,b) \\ x \ne 0}} \frac{x^T A x}{x^T x} = \max_{\substack{P \in \mathcal{P}_{m-1} \\ P \ne 0}} \frac{b^T P^2(A) A b }{b^T P^2(A) b}  = \max_{\substack{P \in \mathcal{P}_{m-1} \\ P \ne 0}} \frac{\tilde b^T P^2(\Lambda) \Lambda \tilde b}{ \tilde b^T P^2(\Lambda) \tilde b},$$
where $\tilde b = Q^T b$. The relative error is given by
$$\frac{\lambda_1(A) - \lambda_1^{(m)}(A,b)}{\lambda_1(A) -\lambda_n(A)}  = \min_{\substack{P \in \mathcal{P}_{m-1} \\ P \ne 0}} \frac{\sum_{i=2}^n \tilde b_i^2 P^2(\lambda_i) (\lambda_1-\lambda_i)}{(\lambda_1-\lambda_n)\sum_{i=1}^n \tilde b_i^2 P^2(\lambda_i)},$$
and the expected $p^{th}$ moment of the relative error is given by
$$\E_{b \sim \mathcal{U}(S^{n-1})} \left[ \left( \frac{\lambda_1 - \lambda_1^{(m)}}{\lambda_1 - \lambda_n}\right)^p \right] = \int_{S^{n-1}}  \min_{\substack{P \in \mathcal{P}_{m-1} \\ P \ne 0}}  \left[ \frac{\sum_{i=2}^n \tilde b_i^2 P^2(\lambda_i) (\lambda_1-\lambda_i)}{(\lambda_1-\lambda_n)\sum_{i=1}^n \tilde b_i^2 P^2(\lambda_i)} \right]^p \, d\sigma(\tilde b),$$
where $\sigma$ is the uniform probability measure on $S^{n-1}$. Because the relative error does not depend on the norm of $\tilde b$ or the sign of any entry, we can replace the integral over $S^{n-1}$ by an integral of $y = (y_1,...,y_n)$ over $[0,\infty)^n$ with respect to the joint chi-square probability density function 
\begin{equation}\label{eqn:chidensity}
f_Y(y) = \frac{1}{(2\pi)^{n/2}} \exp\left\{ -\frac{1}{2} \sum_{i=1}^n y_i \right\} \prod_{i=1}^n y_i^{-\tfrac{1}{2}}
\end{equation}
of $n$ independent chi-square random variables $Y_1,...,Y_n \sim \chi^2_1$ with one degree of freedom each. In particular, we have
\begin{equation}\label{eqn:chi}
\E_{b \sim \mathcal{U}(S^{n-1})} \left[ \left( \frac{\lambda_1 - \lambda_1^{(m)}}{\lambda_1-\lambda_n}\right)^p \right] = \int_{[0,\infty)^n}  \min_{\substack{P \in \mathcal{P}_{m-1} \\ P \ne 0}} \left[ \frac{\sum_{i=2}^n y_i P^2(\lambda_i) (\lambda_1-\lambda_i)}{(\lambda_1-\lambda_n)\sum_{i=1}^n y_i P^2(\lambda_i) } \right]^p f_Y(y) \, dy.
\end{equation}
Similarly, probabilistic estimates with respect to $b \sim \mathcal{U}(S^{n-1})$ can be replaced by estimates with respect to $Y_1,...,Y_n \sim \chi^2_1$, as
\begin{equation}\label{eqn:chi_prob}
\bP_{b \sim \mathcal{U}(S^{n-1})} \left[\frac{\lambda_1 - \lambda_1^{(m)}}{\lambda_1 -\lambda_n} \ge \epsilon \right]  = \bP_{Y_i \sim \chi^2_1} \left[ \min_{\substack{P \in \mathcal{P}_{m-1} \\ P \ne 0}}   \frac{\sum_{i=2}^n Y_i P^2(\lambda_i) (\lambda_1-\lambda_i)}{(\lambda_1-\lambda_n)\sum_{i=1}^n Y_i P^2(\lambda_i) }  \ge \epsilon \right].
\end{equation}
For the remainder of the paper, we almost exclusively use equation (\ref{eqn:chi}) for expected relative error and equation (\ref{eqn:chi_prob}) for probabilistic bounds for relative error. If $P$ minimizes the expression in equation (\ref{eqn:chi}) or (\ref{eqn:chi_prob}) for a given $y$ or $Y$, then any polynomial of the form $\alpha P$, $\alpha \in \mathbb{R}\backslash\{0\}$, is also a minimizer. Therefore, without any loss of generality, we can alternatively minimize over the set $\mathcal{P}_{m-1}(1)= \{ Q \in \mathcal{P}_{m-1} \, | \, Q(1) = 1 \}$. For the sake of brevity, we will omit the subscripts under $\mathbb{E}$ and $\mathbb{P}$ when the underlying distribution is clear from context.
In this work, we make use of asymptotic notation to express the limiting behavior of a function with respect to $n$. A function $f(n)$ is $O(g(n))$ if there exists $M,n_0>0$ such that $|f(n)| \le M g(n)$ for all $n \ge n_0$, $o(g(n))$ if for every $\epsilon>0$ there exists a $n_\epsilon$ such that $|f(n)| \le \epsilon g(x)$ for all $n \ge n_0$, $\omega(g(n))$ if $|g(n)| =o(|f(n)|)$, and $\Theta(g(n))$ if $f(n) =O(g(n))$ and $g(n) = O(f(n))$.


\section{Asymptotic Lower Bounds and Improved Upper Bounds}

In this section, we obtain improved upper bounds for $\E\big[ ((\lambda_1 - \lambda_1^{(m)})/\lambda_1)^p \big]^{1/p}$, $p \ge 1$, and produce nearly-matching lower bounds. In particular, we prove the following theorem.

\begin{theorem}\label{thm:main}
	$$\max_{A \in \mathcal{S}^n}  \bP_{b \sim \mathcal{U}(S^{n-1})} \left[  \frac{\lambda_1(A) - \lambda_1^{(m)}(A,b)}{\lambda_1(A)-\lambda_n(A)} \ge   1-o(1) \right]  \ge 1 - o(1/n)$$
	for $m = o(\ln n)$,
	$$\max_{A \in \mathcal{S}^n}  \bP_{b \sim \mathcal{U}(S^{n-1})} \left[  \frac{\lambda_1(A) - \lambda_1^{(m)}(A,b)}{\lambda_1(A)-\lambda_n(A)} \ge   \frac{.015 \ln^2 n}{ m^2 \ln^2 \ln n} \right]  \ge 1 - o(1/n)$$
	for $m = \Theta(\ln n)$,
	$$ \max_{A \in \mathcal{S}^n}  \bP_{b \sim \mathcal{U}(S^{n-1})} \left[  \frac{\lambda_1(A) - \lambda_1^{(m)}(A,b)}{\lambda_1(A)-\lambda_n(A)} \ge \frac{1.08}{m^2} \right]  \ge 1 - o(1/n)$$ for $m = o\left(n^{1/2} \ln^{-1/2} n\right)$ and $\omega(1)$,
	and
	$$ \max_{A \in \mathcal{S}^n} \E_{b \sim \mathcal{U}(S^{n-1})} \left[ \left( \frac{\lambda_1(A) - \lambda_1^{(m)}(A,b)}{\lambda_1(A) -\lambda_n(A)}\right)^p \right]^{1/p} \le .068 \,  \frac{ \ln^2 \left(n(m-1)^{8p}\right) }{ (m-1)^2}$$
	for $n\ge 100$, $m \ge 10$, $p \ge 1$.
\end{theorem}

\begin{proof}
	The first result is a corollary of \cite[Theorem A.1]{simchowitz2018tight}, and the remaining results follow from Lemmas \ref{lm:msquared_lb}, \ref{lm:lowerbound_mix}, and \ref{lm:upper_rel}.
\end{proof}

By H\"older's inequality, the lower bounds in Theorem \ref{thm:main} also hold for arbitrary $p$-norms, $p \ge 1$. All results are equally applicable to $\lambda_n$, as the Krylov subspace is unaffected by shifting and scaling, namely $\mathcal{K}_m(A,b) = \mathcal{K}_m(\alpha A + \beta I , b)$ for all $\alpha \ne 0$.

We begin by producing asymptotic lower bounds. The general technique is as follows. We choose an infinite sequence of matrices $\{A_n\}_{n=1}^\infty$, $A_n \in \mathcal{S}^n$, treat $m$ as a function of $n$, and show that, as $n$ tends to infinity, for ``most" choices of an initial vector $b$, the relative error of this sequence of matrices is well approximated by an integral polynomial minimization problem for which bounds can be obtained. First, we recall a number of useful propositions regarding Gauss-Legendre quadrature, orthogonal polynomials, and Chernoff bounds for chi-square random variables.

\begin{proposition}{(\cite{forster1990estimates},\cite{tricomi50})} \label{prop:gquad}
	Let $P\in \mathcal{P}_{2k-1}$, $\{x_i\}_{i=1}^k$ be the zeros of the $k^{th}$ degree Legendre polynomial $P_k(x)$, $P_k(1) = 1$, in descending order ($x_1>...>x_k$), and $w_i = 2(1-x_i^2)^{-1} [P'_k(x_i)]^{-2}$, $i = 1,...,k$. Then
	\begin{enumerate}
		\item $ \displaystyle{ \int_{-1}^1 P(x) \, dx = \sum_{i=1}^k w_i P(x_i),}$
		\item $x_i = \left( 1 - \frac{1}{8k^2} \right) \cos \left( \frac{(4i-1) \pi}{4 k + 2} \right) + O(k^{-3})$, $\quad i = 1,...,k$, \medskip
		\item$ \displaystyle{ \frac{\pi}{k+\tfrac{1}{2}} \sqrt{1-x_1^2} \left(1 - \frac{1}{8(k+\tfrac{1}{2})^2(1-x_1^2)} \right) \le w_1 < ... < w_{\big\lfloor\tfrac{k+1}{2}\big\rfloor} \le \frac{\pi}{k+\tfrac{1}{2}}}$.
	\end{enumerate}
\end{proposition}

\begin{proposition}{(\cite[Chapter 7.72]{szeg1939orthogonal}, \cite{tchebichef1883}) }\label{prop:intmax}
	Let $\omega(x) \, dx$ be a measure on $[-1,1]$ with infinitely many points of increase, with orthogonal polynomials $\{p_k(x)\}_{k\ge 0}$, $p_k \in \mathcal{P}_k$. Then
	$$ \max_{\substack{P \in \mathcal{P}_{k} \\ P \ne 0}} \frac{\int_{-1}^1 x P^2(x) \, \omega(x) \, dx }{\int_{-1}^1 P^2(x) \, \omega(x) \, dx } = \max \big\{x \in [-1,1] \, \big| \, p_{k+1}(x) = 0 \big\}.$$
\end{proposition}

\begin{proposition} \label{prop:cher}
	Let $Z \sim \chi^2_k$. Then $\mathbb{P}[Z \le x ] \le \left[\frac{x}{k} \exp\left\{1-\frac{x}{k}\right\} \right]^{\tfrac{k}{2}}$ for $x \le k$ and $ \mathbb{P}[Z \ge x ] \le \left[\frac{x}{k} \exp\left\{1-\frac{x}{k}\right\} \right]^{\tfrac{k}{2}}$ for $x \ge k$.
\end{proposition}

\begin{proof}
	The result follows from taking \cite[Lemma 2.2]{dasgupta2003elementary} and letting $d \rightarrow \infty$.
\end{proof}

We are now prepared to prove a lower bound of order $m^{-2}$.

\begin{lemma} \label{lm:msquared_lb}
	If $m = o\big(n^{1/2} \ln^{-1/2} n \big)$ and $\omega(1)$, then
	$$\max_{A \in \mathcal{S}^n}  \bP_{b \sim \mathcal{U}(S^{n-1})} \left[  \frac{\lambda_1(A) - \lambda_1^{(m)}(A,b)}{\lambda_1(A) - \lambda_n(A)} \ge \frac{1.08}{m^2} \right]  \ge 1 - o(1/n).$$
\end{lemma}

\begin{proof}
	The structure of the proof is as follows. We choose a matrix with eigenvalues based on the zeros of a Legendre polynomial, and show that a large subset of $[0,\infty)^n$ (with respect to $f_Y(y)$) satisfies conditions that allow us to lower bound the relative error using an integral minimization problem. Let $x_1>...>x_{2m}$ be the zeros of the $(2m)^{th}$ degree Legendre polynomial. Let $A \in \mathcal{S}^n$ have eigenvalue $x_1$ with multiplicity one, and remaining eigenvalues given by $x_j$, $j = 2,...,2m$, each with multiplicity at least $\lfloor (n-1) /(2m-1) \rfloor$.
	By equation (\ref{eqn:chi_prob}),
	\begin{eqnarray*}
		\bP \left[\frac{\lambda_1 - \lambda_1^{(m)}}{\lambda_1-\lambda_n} \ge \epsilon \right]  &=& \bP \left[  \min_{\substack{P \in \mathcal{P}_{m-1} \\ P \ne 0}}  \frac{ \sum_{i=2}^n Y_i P^2(\lambda_i) (\lambda_1-\lambda_i)}{(\lambda_1-\lambda_n)\sum_{i=1}^n Y_i P^2(\lambda_i) }  \ge \epsilon \right] \\
		&=&  \bP \left[   \min_{\substack{P \in \mathcal{P}_{m-1} \\ P \ne 0}} \frac{\sum_{j=2}^{2m} \hat Y_j P^2(x_j) \left(x_1 -x_j \right) }{(x_1-x_{2m})\sum_{j=1}^{2m} \hat Y_j P^2\left(x_j\right) } \ge \epsilon \right],
	\end{eqnarray*}
	where $Y_1,...,Y_n \sim \chi^2_1$, and $\hat Y_j$ is the sum of the $Y_i$ that satisfy $\lambda_i= x_j$. $\hat Y_1$ has one degree of freedom, and $\hat Y_j$, $j = 2,...,2m$, each have at least $\lfloor (n-1)/(2m-1) \rfloor$ degrees of freedom. Let $w_1,...,w_{2m}$ be the weights of Gaussian quadrature associated with $x_1,...,x_{2m}$. By Proposition \ref{prop:gquad},
	$$x_1 = 1 - \frac{4 + 9\pi^2}{128m^2} + O(m^{-3}), \qquad x_2 = 1 - \frac{4 + 49\pi^2}{128 m^2}+ O(m^{-3}),$$
	%
	and, therefore, $1-x_1^2 = \frac{4 + 9 \pi^2}{64m^2} + O(m^{-3})$. Again, by Proposition \ref{prop:gquad}, we can lower bound the smallest ratio between weights by
	\begin{eqnarray*}
		\frac{w_1}{w_m} &\ge& \sqrt{1-x_1^2}\left(1 - \frac{1}{8(2m+1/2)^2 (1-x_1^2)} \right)\\
		&=&\left( \frac{\sqrt{4 + 9 \pi^2}}{8m}  +O(m^{-2}) \right) \left(1 - \frac{1}{(4 + 9\pi^2)/2 +O(m^{-1})} \right) \\
		&=& \frac{2+9 \pi^2}{8 \sqrt{4 + 9\pi^2} \, m} +O(m^{-2}).
	\end{eqnarray*}
	Therefore, by Proposition \ref{prop:cher},
	\begin{eqnarray*} 
		\bP \left[ \min_{j\ge 2} \hat Y_j \ge \frac{w_m }{w_1} \hat Y_1 \right] &\ge& \mathbb{P}\left[ \hat Y_j \ge \frac{1}{3} \left\lfloor \frac{n-1}{2m-1} \right\rfloor, \, j \ge 2\right] \times \mathbb{P}\left[ \hat Y_1 \le \frac{w_1}{3 w_m} \left\lfloor \frac{n-1}{2m-1} \right\rfloor \right] \\
		&\ge& \bigg[ 1 - \left(\tfrac{e^{2/3}}{3}\right)^{ \tfrac{1}{2} \left\lfloor \tfrac{n-1}{2m-1} \right\rfloor} \bigg]^{2m-1}  \bigg[ 1- \left(\tfrac{w_1}{3 w_m} \left\lfloor \tfrac{n-1}{2m-1} \right\rfloor e^{1-\tfrac{w_1}{3 w_m} \left\lfloor \tfrac{n-1}{2m-1} \right\rfloor}  \right)^{\tfrac{1}{2}}  \bigg]\\
		&=& 1 - o(1/n).
	\end{eqnarray*}
	We now restrict our attention to values of $Y = (Y_1,...,Y_n) \in [0,\infty)^n$ that satisfy $w_1 \min_{j\ge 2} \hat Y_j \ge  w_m  \hat Y_1$. If, for some fixed choice of $Y$,
	\begin{equation}\label{eqn:lbproof}
	\min_{\substack{P \in \mathcal{P}_{m-1} \\ P \ne 0}}   \frac{\sum_{j=2}^{2m} \hat Y_j P^2(x_j) (x_1-x_j) }{\sum_{j=1}^{2m} \hat Y_j P^2 (x_j) } \le x_1 - x_2,
	\end{equation}
	then, by Proposition \ref{prop:gquad} and Proposition \ref{prop:intmax} for $\omega(x) =1$,
	\begin{eqnarray*}
		\min_{\substack{P \in \mathcal{P}_{m-1} \\ P \ne 0}}   \frac{\sum_{j=2}^{2m} \hat Y_j P^2(x_j) (x_1-x_j) }{\sum_{j=1}^{2m} \hat Y_j P^2(x_j) } 	&\ge& \min_{\substack{P \in \mathcal{P}_{m-1} \\ P \ne 0}}  \frac{\sum_{j=2}^{2m}w_j P^2(x_j) (x_1-x_j) }{\sum_{j=1}^{2m} w_j P^2(x_j)} \\ &=& \min_{\substack{P \in \mathcal{P}_{m-1} \\ P \ne 0}} \frac{\sum_{j=1}^{2m}w_j P^2(x_j)(1-x_j) }{\sum_{j=1}^{2m} w_j P^2(x_j)} - (1-x_1)\\
		&=& \min_{\substack{P \in \mathcal{P}_{m-1} \\ P \ne 0}} \frac{\displaystyle{\int_{-1}^1 P^2(y) (1-y) \, dy}}{\displaystyle{\int_{-1}^1 P^2(y)  \, dy}} - \frac{4 + 9 \pi^2}{128m^2} +O(m^{-3})\\
		&=& \frac{4 + 9 \pi^2}{32m^2} - \frac{4 + 9 \pi^2}{128m^2} +O(m^{-3}) = \frac{12 + 27 \pi^2}{128m^2} +O(m^{-3}).
	\end{eqnarray*}
	Alternatively, if equation (\ref{eqn:lbproof}) does not hold, then we can lower bound the left hand side of (\ref{eqn:lbproof}) by $x_1-x_2 = \tfrac{5 \pi^2}{16 m^2} +O(m^{-3})$. Noting that $x_1 - x_{2m} \le 2$ completes the proof, as $  \tfrac{12 + 27\pi^2}{256}$ and $\tfrac{5 \pi^2}{32}$ are both greater than $1.08$.
\end{proof}

The previous lemma illustrates that the inverse quadratic dependence of the relative error on the number of iterations $m$ persists up to $m =o(n^{1/2}/\ln^{1/2}n)$. As we will see in Sections 4 and 6, this $m^{-2}$ error is indicative of the behavior of the Lanczos method for a wide range of typical matrices encountered in practice. Next, we aim to produce a lower bound of the form $\ln^2 n / (m^2 \ln^2 \ln n)$. To do so, we will make use of more general Jacobi polynomials instead of Legendre polynomials. However, due to the distribution of eigenvalues required to produce this bound, Gaussian quadrature is no longer exact, and we must make use of estimates for basic composite quadrature. We recall the following propositions regarding composite quadrature and Jacobi polynomials.

\begin{proposition} \label{prop:rsum}
	If $f \in \mathcal{C}^1[a,b]$, then for $a=x_0<...<x_n = b$,
	$$ \left| \int_{a}^b f(x) \, dx - \sum_{i=1}^n (x_i - x_{i-1})f(x_i^*) \right| \le \frac{b - a}{2} \max_{x \in [a,b]} | f'(x)| \, \max_{i=1,...,n} (x_i - x_{i-1}),$$
	where $x_i^* \in [x_{i-1},x_i]$, $i = 1,...,n$.
\end{proposition}

\begin{proposition}{(\cite[Chapter 3.2]{shen2011spectral})} \label{prop:jac2}
	Let $\{P^{(\alpha,\beta)}_k(x)\}_{k=0}^\infty$, $\alpha, \beta >-1$, be the orthogonal system of Jacobi polynomials over $[-1,1]$ with respect to weight function $\omega^{\alpha,\beta}(x) = (1-x)^\alpha(1+x)^\beta$, namely,
	$$P^{(\alpha,\beta)}_k(x) =  \frac{\Gamma(k+\alpha+1)}{k! \, \Gamma(k+\alpha+\beta+1)} \sum_{i=0}^k {k \choose i } \frac{\Gamma(k+i+\alpha+\beta+1)}{\Gamma(i+\alpha+1)} \left( \frac{x-1}{2} \right)^i.$$
	Then
	\begin{enumerate}[(i)]
		\item $\displaystyle{\int_{-1}^1 \left[ P^{(\alpha,\beta)}_k(x) \right]^2 \omega^{\alpha,\beta}(x) \, dx = \frac{2^{\alpha + \beta +1} \Gamma(k+\alpha+1) \Gamma(k+\beta+1)}{(2k+\alpha + \beta +1)k! \, \Gamma(k+\alpha + \beta+1)} }$, \\
		\item $\displaystyle{ \max_{x \in [-1,1]} \left|P^{(\alpha,\beta)}_k(x) \right| =  \max \left\{ \frac{\Gamma(k+\alpha+1)}{k! \, \Gamma(\alpha+1)}, \frac{\Gamma(k+\beta+1)}{k! \, \Gamma(\beta+1)} \right\}}$ for $\max \{\alpha ,\beta\} \ge -\tfrac{1}{2}$, \\
		\item $\displaystyle{\frac{d}{dx} P^{(\alpha,\beta)}_k(x) = \frac{k+\alpha + \beta +1}{2} P_{k-1}^{(\alpha+1,\beta+1)}(x) }$, \\
		\item $\displaystyle{\max \, \{ x \in [-1,1] \, | \, P^{(\alpha,\beta)}_k(x) = 0 \} \le \sqrt{1 - \left(\frac{\alpha+3/2}{k + \alpha + 1/2}\right)^2 }}$ for $\alpha \ge \beta >-\tfrac{11}{12}$ and $k >1$.

	\end{enumerate}
\end{proposition}

\begin{proof}
	Equations (i)-(iii) are standard results, and can be found in \cite[Chapter 3.2]{shen2011spectral} or \cite{szeg1939orthogonal}. What remains is to prove (iv). By \cite[Lemma 3.5]{nikolov2003inequalities}, the largest zero $x_1$ of the $k^{th}$ degree Gegenbauer polynomial (i.e., $P^{(\lambda-1/2,\lambda-1/2)}_k(x)$), with $\lambda >- 5/12$, satisfies
	$$x_1^2 \le \frac{(k-1)(k+2\lambda+1)}{(k+\lambda)^2 + 3 \lambda + \tfrac{5}{4} + 3 (\lambda+\tfrac{1}{2})^2/(k-1)} \le \frac{(k-1)(k+2 \lambda +1)}{(k+\lambda)^2}  =  1 - \left(\frac{\lambda +1}{k+\lambda}\right)^2 .$$
	By \cite[Theorem 2.1]{driver2008interlacing}, the largest zero of $P^{(\alpha,\beta+t)}_k(x)$ is strictly greater than the largest zero of $P^{(\alpha,\beta)}_k(x)$ for any $t >0$. As $\alpha \ge \beta$, combining these two facts provides our desired result.
\end{proof}

For the sake of brevity, the inner product and norm on $[-1,1]$ with respect to  $\omega^{\alpha,\beta}(x) =(1-x)^\alpha (1+x)^\beta$ will be denoted by $\langle \cdot, \cdot \rangle_{\alpha,\beta}$ and $\| \cdot \|_{\alpha,\beta}$. We are now prepared to prove the following lower bound.


%
%
%
%
%

\begin{lemma}\label{lm:lowerbound_mix}
	If $m = \Theta\left(\ln n\right)$, then
	$$\max_{A \in \mathcal{S}^n}  \bP_{b \sim \mathcal{U}(S^{n-1})} \left[  \frac{\lambda_1(A) - \lambda_1^{(m)}(A,b)}{\lambda_1(A)-\lambda_n(A)} \ge  \frac{.015 \ln^2 n}{ m^2 \ln^2 \ln n} \right]  \ge 1 - o(1/n).$$
\end{lemma}

\begin{proof}
	The structure of the proof is similar in concept to that of Lemma \ref{lm:msquared_lb}. We choose a matrix with eigenvalues based on a function corresponding to an integral minimization problem, and show that a large subset of $[0,\infty)^n$ (with respect to $f_Y(y)$) satisfies conditions that allows us to lower bound the relative error using the aforementioned integral minimization problem. The main difficulty is that, to obtain improved bounds, we must use Proposition \ref{prop:intmax} with a weight function that requires quadrature for functions that are no longer polynomials, and, therefore, cannot be represented exactly using Gaussian quadrature. In addition, the required quadrature is for functions whose derivative has a singularity at $x = 1$, and so Proposition \ref{prop:rsum} is not immediately applicable. Instead, we perform a two-part error analysis. In particular, if a function is $C^1[0,a]$, $a <1$, and monotonic on $[a,1]$, then using Proposition \ref{prop:rsum} on $[0,a]$ and a monotonicity argument on $[a,1]$ results in an error bound for quadrature.
	
	Let $\ell = \left\lfloor \frac{.2495 \ln n}{ \ln \ln n}\right\rfloor$, $k = \left\lfloor m^{4.004 \ell} \right\rfloor$, and $m = \Theta(\ln n)$.  Consider a matrix $A \in \mathcal{S}^n$ with eigenvalues given by $ f(x_{j})$, $j = 1,...,k$, each with multiplicity either $\lfloor n/k \rfloor$ or $\lceil n/k \rceil$, where $f(x) = 1 - 2(1-x)^{\frac{1}{\ell}}$, and $x_j = 
	j/k$, $j = 1,...,k$. Then
	\begin{eqnarray*}
		\bP \left[ \frac{\lambda_1 - \lambda_1^{(m)}}{\lambda_1-\lambda_n} \ge \epsilon \right] &=& \bP \left[   \min_{\substack{P \in \mathcal{P}_{m-1} \\ P \ne 0}}  \frac{\sum_{i=2}^{n} Y_i P^2\left(\lambda_i\right) \left(\lambda_1 - \lambda_i  \right) }{(\lambda_1-\lambda_n)\sum_{i=1}^{n}  Y_i P^2\left(\lambda_i\right) } \ge \epsilon \right]\\
		&\ge& \bP \left[    \min_{\substack{P \in \mathcal{P}_{m-1} \\ P \ne 0}}   \frac{\sum_{j=1}^{k} \hat Y_j P^2\left(f(x_j)\right) \left( 1-f(x_j)\right) }{2\sum_{j=1}^{k} \hat Y_j P^2\left(f(x_j)\right) } \ge \epsilon \right],
	\end{eqnarray*}
	where $Y_1,...,Y_n \sim \chi^2_1$, and $\hat Y_j$ is the sum of the $Y_i$'s that satisfy $\lambda_i = f(x_j)$. Each $\hat Y_j$, $j = 1,...,k$, has either $\lfloor n/k \rfloor$ or $\lceil n/k \rceil$ degrees of freedom. Because $m = \Theta(
	\ln n)$, we have $k = o(n^{.999})$ and, by Proposition \ref{prop:cher},
	\begin{eqnarray*}
		\mathbb{P}\left[ .999 \lfloor \tfrac{n}{k}  \rfloor  \le  \hat Y_j \le 1.001 \lceil \tfrac{n}{k}  \rceil, \, j = 1,...,k  \right] &\ge& \left(1- \left(.999e^{.001}\right)^{\lfloor \tfrac{n}{k}  \rfloor} - \left(1.001e^{-.001}\right)^{\lceil \tfrac{n}{k}  \rceil} \right)^{k} \\ &=& 1-o(1/n) .
	\end{eqnarray*}
	Therefore, with probability $1 - o(1/n)$,
	$$  \min_{\substack{P \in \mathcal{P}_{m-1} \\ P \ne 0}}  \frac{\sum_{j=1}^{k} \hat Y_j P^2\left(f(x_j)\right) \left( 1-f(x_j)\right) }{2\sum_{j=1}^{k} \hat Y_j P^2\left(f(x_j)\right) } \ge \frac{.998}{2}  \min_{\substack{P \in \mathcal{P}_{m-1} \\ P \ne 0}} \frac{\sum_{j=1}^k P^2(f(x_j))(1-f(x_j))}{\sum_{j=1}^k P^2(f(x_j))}.$$
	
	Let $ P(y) = \sum_{r=0}^{m-1} a_r P_r^{(\ell,0)}(y)$, $a_r \in \mathbb{R}$, $r = 0,...,m-1$, and define 
	$$g_{r,s}(x) = P_r^{(\ell,0)}(f(x))P_s^{(\ell,0)}(f(x)) \quad  \text{ and } \quad \hat g_{r,s}(x) = g_{r,s}(x) (1-f(x)),$$
	$ r,s = 0,...,m-1$. Then we have
	$$ \frac{\sum_{j=1}^k P^2(f(x_j))(1-f(x_j))}{\sum_{j=1}^k P^2(f(x_j))} =\frac{\sum_{r,s=0}^{m-1} a_r a_s \sum_{j=1}^k \hat g_{r,s}(x_j)}{\sum_{r,s=0}^{m-1} a_r a_s \sum_{j=1}^k g_{r,s}(x_j)}.$$
	We now replace each quadrature $\textstyle \sum_{j=1}^k \hat g_{r,s}(x_j)$ or $\textstyle \sum_{j=1}^k g_{r,s}(x_j)$ in the previous expression by its corresponding integral, plus a small error term. The functions $g_{r,s}$ and $\hat g_{r,s}$ are not elements of $\mathcal{C}^1[0,1]$, as $f'(x)$ has a singularity at $x=1$, and, therefore, we cannot use Proposition \ref{prop:rsum} directly. Instead we break the error analysis of the quadrature into two components. We have $g_{r,s},\hat g_{r,s} \in \mathcal{C}^1[0,a]$ for any $0<a<1$, and, if $a$ is chosen to be large enough, both $g_{r,s}$ and $\hat g_{r,s}$ will be monotonic on the interval $[a,1]$. In that case, we can apply Proposition \ref{prop:rsum} to bound the error over the interval $[0,a]$ and use basic properties of Riemann sums of monotonic functions to bound the error over the interval $[a,1]$.
	
	The function $f(x)$ is increasing on $[0,1]$, and, by equations (iii) and (iv) of Proposition \ref{prop:jac2}, the function $P_r^{(\ell,0)}(y)$, $r=2,...,m-1,$ is increasing on the interval
	\begin{equation}\label{eqn:preinterval}
	\left[\sqrt{1 - \left(\frac{\ell+5/2}{m + \ell - 1/2}\right)^2 }\, , \; 1\right] .
	\end{equation}
	The functions $P_0^{(\ell,0)}(y) = 1$ and $P_1^{(\ell,0)}(y) =\ell +1 + (\ell+2)(y-1)/2$ are clearly non-decreasing over interval (\ref{eqn:preinterval}). By the inequality $ \sqrt{1-y} \le 1-y/2 $, $y \in [0,1]$, the function $P_r^{(\ell,0)}(f(x))$, $r=0,...,m-1,$ is non-decreasing on the interval
	\begin{equation}\label{eqn:interval}
	\left[ 1 - \left(\frac{\ell+5/2}{2m + 2\ell - 1}\right)^{2\ell} \, , \; 1 \right].
	\end{equation}
	Therefore, the functions $g_{r,s}(x)$ are non-decreasing and $\hat g_{r,s}(x)$ are  non-increasing on the interval (\ref{eqn:interval}) for all $r,s=0,...,m-1$. The term $\textstyle \left(\frac{\ell+5/2}{2m + 2\ell - 1}\right)^{2\ell} = \omega(1/k)$, and so, for sufficiently large $n$, there exists an index $j^* \in \{1,...,k-2\}$ such that
	$$   1 - \left(\frac{\ell+5/2}{2m + 2\ell - 1}\right)^{2\ell} \le \,  x_{j^*} \, <  1 - \left(\frac{\ell+5/2}{2m + 2\ell - 1}\right)^{2\ell}  + \frac{1}{k} \, < \,  1.$$
	
	We now upper bound the derivatives of $g_{r,s}$ and $\hat g_{r,s}$ on the interval $[0,x_{j^*}]$. By equation (iii) of Proposition \ref{prop:jac2}, the first derivatives of $g_{r,s}$ and $\hat g_{r,s}$ are given by
	\begin{eqnarray*}
		g_{r,s}'(x) &=&  f'(x) \left(  \left[\frac{d}{dy}  P_r^{(\ell,0)}(y)\right]_{y = f(x)} P_s^{(\ell,0)}(f(x))  + P_r^{(\ell,0)}(f(x))   \left[\frac{d}{dy}  P_s^{(\ell,0)}(y)\right]_{y = f(x)} \right) \\
		&=& \frac{(1-x)^{-\tfrac{\ell-1}{\ell}} }{\ell} \bigg( (r+\ell+1) P_{r-1}^{(\ell+1,1)}(f(x)) P_s^{(\ell,0)}(f(x)) + (s+\ell+1) P_r^{(\ell,0)}(f(x)) P_{s-1}^{(\ell+1,1)}(f(x))    \bigg),
	\end{eqnarray*}
	and
	$$ \hat g_{r,s}'(x)  = g_{r,s}'(x) (1-x)^{\tfrac{1}{\ell}} -2 g_{r,s}(x)  \frac{(1-x)^{-\tfrac{\ell-1}{\ell}} }{\ell}.$$
	By equation (ii) of Proposition \ref{prop:jac2}, and the inequality ${x \choose y } \le \left( \frac{e x}{y} \right)^y$, $x,y \in \mathbb{N}$, $y \le x$,
	\begin{eqnarray*}
		\max_{x \in[0,x_{j^*}]} | g'_{r,s}(x)| &\le& \frac{(1-x_{j^*})^{-\tfrac{\ell-1}{\ell}} }{\ell}  \bigg( (r+\ell+1) {r+\ell \choose \ell+1} {s+\ell \choose \ell}   +  (s+\ell+1) {r+\ell \choose \ell}   {s+\ell \choose \ell+1}   \bigg) \\
		&\le& 2 \, \frac{m+\ell}{\ell} \left(\left( \frac{\ell+5/2}{2m + 2\ell - 1}\right)^{2\ell} - \frac{1}{k} \right)^{-\frac{\ell-1}{\ell}} \left( \frac{e(m+\ell-1)}{\ell} \right)^{2\ell+1},
	\end{eqnarray*}
	and
	\begin{eqnarray*}
		\max_{x \in [0,x_{j^*}]} \left| \hat g'_{r,s}(x) \right|  &\le&  \max_{x \in[0,x_{j^*}]} | g'_{r,s}(x)| + 2\frac{(1-x_{j^*})^{-\tfrac{\ell-1}{\ell}} }{\ell} {r+\ell \choose \ell}   {s+\ell \choose \ell} \\
		&\le& 4 \, \frac{m+\ell}{\ell} \left(\left( \frac{\ell+5/2}{2m + 2\ell - 1}\right)^{2\ell} - \frac{1}{k} \right)^{-\frac{\ell-1}{\ell}} \left( \frac{e(m+\ell-1)}{\ell} \right)^{2\ell+1}.
	\end{eqnarray*}
	Therefore, both $\max_{x \in[0,x_{j^*}]} | g'_{r,s}(x)|$ and $\max_{x \in [0,x_{j^*}]} | \hat g'_{r,s}(x)| $ are $o(m^{4.002 \ell})$. Then, by Proposition \ref{prop:rsum}, equation (ii) of Proposition \ref{prop:jac2} and monotonicity on the interval $[x_{j^*},1]$, we have
	\begin{eqnarray*}
		\left| \frac{1}{k} \sum_{j=1}^k g_{r,s}(x_j) - \int_{0}^1  g_{r,s}(x) \, dx \right| &\le& \left| \frac{1}{k} \sum_{j=1}^{j^*} g_{r,s}(x_j) - \int_{0}^{x_{j^*}}  g_{r,s}(x) \, dx \right|  + \left| \frac{1}{k} \sum_{j = j^* +1}^{k}  g_{r,s}(x_j) - \int_{a}^1 g_{r,s}(x) \, dx \right| \\
		&\le& \frac{1+o(1)}{2 k m^{4.002\ell}} + \frac{ g_{r,s}(1)}{k} \, \le \, \frac{1+o(1)}{2 m^{.002\ell}} + \frac{ 1}{k} \left( \frac{e(m+\ell+1)}{\ell+1} \right)^{2(\ell+1)} \, \le \, \frac{1+o(1)}{2 m^{.002 \ell}},
	\end{eqnarray*}
	and, similarly,
	\begin{eqnarray*}
		\left| \frac{1}{k} \sum_{j=1}^k \hat g_{r,s}(x_j) - \int_{0}^1 \hat g_{r,s}(x) \, dx \right| &\le& \left| \frac{1}{k} \sum_{j=1}^{j^*} \hat g_{r,s}(x_j) - \int_{0}^{x_{j^*}} \hat g_{r,s}(x) \, dx \right|  + \left| \frac{1}{k} \sum_{j=j^*+1}^{k} \hat g_{r,s}(x_j) - \int_{x_{j^*}}^1 \hat g_{r,s}(x) \, dx \right| \\
		&\le& \frac{1+o(1)}{2 m^{.002\ell}} +\frac{ \hat g_{r,s}(x_{j^*})}{k} \, \le \, \frac{1+o(1)}{2 m^{.002\ell}} +\frac{ g_{r,s}(1)}{k} \, \le \,  \frac{1+o(1)}{ 2 m^{.002 \ell}}.
	\end{eqnarray*}
	Let us denote this upper bound by $M = (1+o(1))/(2m^{.002 \ell})$. By using the substitution $x = 1 - (\tfrac{1-y}{2})^{\ell}$, we have
	$$	\int_{0}^1 \hat g_{r,s}(x) \,  dx = \frac{\ell}{2^{\ell}} \langle P_r^{(\ell,0)}, P_s^{(\ell,0)} \rangle_{\ell,0} \quad \text{and} \quad \int_{0}^1 g_{r,s}(x) \,  dx = \frac{\ell}{2^{\ell}} \langle P_r^{(\ell,0)}, P_s^{(\ell,0)} \rangle_{\ell-1,0} .$$
	Then
	$$ \max_{\substack{P \in \mathcal{P}_{m-1} \\ P \ne 0}} \frac{\sum_{j=1}^k P^2(f(x_j))}{\sum_{j=1}^k P^2(f(x_j))(1-f(x_j))} \le \max_{ \substack{a \in \mathbb{R}^{m} \backslash 0 \\ |\epsilon_{r,s} | \le M \\ |\hat \epsilon_{r,s} | \le M}} \frac{\sum_{r,s=0}^{m-1} a_r a_s \left[ \frac{\ell}{2^{\ell}} \langle P_r^{(\ell,0)}, P_s^{(\ell,0)} \rangle_{\ell-1,0} + \epsilon_{r,s} \right]}{\sum_{r,s=0}^{m-1} a_r a_s \left[ \frac{\ell}{2^{\ell}} \langle P_r^{(\ell,0)}, P_s^{(\ell,0)} \rangle_{\ell,0} +\hat \epsilon_{r,s} \right] }.$$
	
	Letting $\tilde a_r = a_r \|P_r^{(\ell,0)} \|_{\ell,0}$, $r = 0,...,m-1$, and noting that, by equation (i) of Proposition \ref{prop:jac2}, $\|P_r^{(\ell,0)} \|_{\ell,0}=\left( \frac{2^{\ell+1}}{2r+\ell+1} \right)^{1/2}$, we obtain the bound
	$$ \max_{\substack{P \in \mathcal{P}_{m-1} \\ P \ne 0}} \frac{\sum_{j=1}^k P^2(f(x_j))}{\sum_{j=1}^k P^2(f(x_j))(1-f(x_j))} \le \max_{ \substack{\tilde a \in \mathbb{R}^{m} \backslash 0 \\ |\epsilon_{r,s} | \le \epsilon \\ |\hat \epsilon_{r,s} | \le \epsilon}}  \frac{\sum_{r,s=0}^{m-1} \tilde a_r \tilde a_s \left[ \frac{\langle P_r^{(\ell,0)}, P_s^{(\ell,0)} \rangle_{\ell-1,0} }{ \|P_r^{(\ell,0)}\|_{\ell,0} \|P_s^{(\ell,0)}\|_{\ell,0}  }  +  \epsilon_{r,s} \right]}{\sum_{r=0}^{m-1} \tilde a_r^2 + \sum_{r,s=0}^{m-1} \tilde a_r \tilde a_s  \hat \epsilon_{r,s}},$$
	where
	$$\epsilon = \frac{M(2m+\ell-1)}{\ell} = (1+o(1)) \frac{2 m +\ell-1}{2 m^{.002\ell}\ell} = o(1/m).$$
	Let $B \in \mathbb{R}^{m \times m}$ be given by $B(r,s) =\langle P_r^{(\ell,0)},P_s^{(\ell,0)}\rangle_{\ell-1,0} \|P_r^{(\ell,0)} \|^{-1}_{\ell,0}\|P_s^{(\ell,0)} \|^{-1}_{\ell,0}$, $r,s = 0,...,m-1$. By Proposition \ref{prop:intmax} applied to $\omega^{\ell-1,0}(x)$ and equation (iv) of Proposition \ref{prop:jac2}, we have
	$$ \max_{ \substack{\tilde a \in \mathbb{R}^{m} \backslash 0 }}  \frac{\tilde a^T B \tilde a}{\tilde a^T \tilde a}   \le \frac{ 1}{1-\sqrt{1-\left(\frac{\ell+1/2}{m + \ell - 1/2}\right)^2}} \le 2 \left(\frac{m + \ell - 1/2}{\ell+1/2}\right)^2.$$

	This implies that
	\begin{eqnarray*}
		\max_{\substack{P \in \mathcal{P}_{m-1} \\ P \ne 0}} \frac{\sum_{j=1}^k P^2(f(x_j))}{\sum_{j=1}^k P^2(f(x_j))(1-f(x_j))}  &\le& \max_{\substack{\tilde a \in \mathbb{R}^{m} \\ \tilde a \ne 0}}  \frac{\tilde a^T \left(B + \epsilon \mathds{1} \mathds{1}^T \right) \tilde a}{\tilde a^T \left( I - \epsilon \mathds{1} \mathds{1}^T \right) \tilde a }  \\ &\le& \frac{1}{1-\epsilon m} \max_{\substack{\tilde a \in \mathbb{R}^{m} \\ \tilde a \ne 0 }}  \left[  \frac{\tilde a^T B \tilde a}{ \tilde a^T \tilde a}\right] + \frac{\epsilon m}{1-\epsilon m} \\
		&\le& \frac{2}{1-\epsilon m} \left(\frac{m + \ell - 1/2}{\ell+1/2}\right)^2+ \frac{\epsilon m}{1-\epsilon m},
	\end{eqnarray*}
	and that, with probability $1-o(1/n)$, 
	$$	\min_{\substack{P \in \mathcal{P}_{m-1} \\ P \ne 0}}  \frac{\sum_{j=1}^{k} \hat Y_j P^2\left(f(x_j)\right) \left( 1-f(x_j)\right) }{\sum_{j=1}^{k} \hat Y_j P^2\left(f(x_j)\right) } \ge  (1-o(1))  \frac{.998}{4} \left(\frac{\ell+1/2}{m + \ell - 1/2}\right)^2 \ge \frac{.015 \ln^2 n}{m^2 \ln^2 \ln n} .$$
	This completes the proof.
\end{proof}

In the previous proof, a number of very small constants were used, exclusively for the purpose of obtaining an estimate with constant as close to $1/64$ as possible. These constants can be replaced by more practical numbers (that begin to exhibit asymptotic convergence for reasonably sized $n$), at the cost of a mildly worse constant. This completes the analysis of asymptotic lower bounds. 

We now move to upper bounds for relative error in the $p$-norm. Our estimate for relative error in the one-norm is of the same order as the estimate (\ref{eqn:kw}), but with an improved constant. Our technique for obtaining these estimates differs from the technique of \cite{kuczynski1992estimating} in one key way. Rather than integrating first by $b_1$ and using properties of the arctan function, we replace the ball $B^n$ by $n$ chi-square random variables on $[0,\infty)^n$, and iteratively apply Cauchy-Schwarz to our relative error until we obtain an exponent $c$ for which the inverse chi-square distribution with one degree of freedom has a convergent $c^{th}$ moment.

\begin{lemma} \label{lm:upper_rel}
	Let $n \ge 100$, $m \ge 10$, and $p \ge 1$. Then
	$$\max_{A \in \mathcal{S}^n} \E_{b \sim \mathcal{U}(S^{n-1})} \left[ \left( \frac{\lambda_1(A) - \lambda_1^{(m)}(A,b)}{\lambda_1(A)-\lambda_n(A)}\right)^p \right]^{\frac{1}{p}}  \le .068 \, \frac{ \ln^2 \left(n(m-1)^{8p}\right) }{(m-1)^2}.$$
\end{lemma}

\begin{proof}
	Without loss of generality, suppose $\lambda_1(A) =1$ and $\lambda_n(A) = 0$. By repeated application of Cauchy Schwarz,
	$$\frac{\sum_{i=1}^n y_i Q^2(\lambda_i)(1 - \lambda_i)}{ \sum_{i=1}^n y_i Q^2(\lambda_i)} \le \left[ \frac{\sum_{i=1}^n y_i Q^2(\lambda_i)(1 - \lambda_i)^2}{ \sum_{i=1}^n y_i Q^2(\lambda_i)} \right]^{\frac{1}{2}} \le \left[  \frac{\sum_{i=1}^n y_i Q^2(\lambda_i)(1 - \lambda_i)^{2^q}}{ \sum_{i=1}^n y_i Q^2(\lambda_i)} \right]^{\frac{1}{2^q}}$$
	for $q \in \mathbb{N}$. Choosing $q$ to satisfy $2p < 2^q \le 4p$, and using equation (\ref{eqn:chi}) with polynomial normalization $Q(1) = 1$, we have
	\begin{eqnarray*}
		\E \left[ \left( \frac{\lambda_1 - \lambda_1^{(m)}}{\lambda_1-\lambda_n}\right)^p \right]^{\frac{1}{p}} &\le& \min_{Q \in \mathcal{P}_{m-1}(1)} \left[ \int_{[0,\infty)^n} \left( \frac{\sum_{i=2}^n y_i Q^2(\lambda_i)(1 - \lambda_i)^{2^q}}{ \sum_{i=1}^n y_i Q^2(\lambda_i)} \right)^{\frac{p}{2^q}} f_Y(y) \, dy \right]^{\frac{1}{p}} \\
		&\le& \min_{Q \in \mathcal{P}_{m-1}(1)} \left[ \int_{[0,\infty)^n} \left( \frac{\sum_{i: \lambda_i < \beta} y_i Q^2(\lambda_i)(1 - \lambda_i)^{2^q}}{ \sum_{i=1}^n y_i Q^2(\lambda_i)} \right)^{\frac{p}{2^q}} f_Y(y) \, dy \right]^{\frac{1}{p}} \\
		&&\qquad + \quad \left[ \int_{[0,\infty)^n} \left( \frac{\sum_{i: \lambda_i \ge \beta} y_i Q^2(\lambda_i)(1 - \lambda_i)^{2^q}}{ \sum_{i=1}^n y_i Q^2(\lambda_i)} \right)^{\frac{p}{2^q}} f_Y(y) \, dy \right]^{\frac{1}{p}} \\
	\end{eqnarray*}
	for any $ \beta \in (0,1)$. The integrand of the first term satisfies
	\begin{eqnarray*}
		\left( \frac{\sum_{i: \lambda_i < \beta} y_i Q^2(\lambda_i)(1 - \lambda_i)^{2^q}}{ \sum_{i=1}^n y_i Q^2(\lambda_i)} \right)^{\frac{p}{2^q}} &\le&  \left( \frac{\sum_{i: \lambda_i < \beta} y_i Q^2(\lambda_i)(1 - \lambda_i)^{2^q}}{ \sum_{i=1}^n y_i Q^2(\lambda_i)} \right)^{\frac{1}{4}}  \\
		&\le& \max_{x \in [0,\beta]} |Q(x)|^{1/2}(1 - x)^{2^{q-2}}  \left(\frac{\sum_{i:\lambda_i<\beta} y_i}{y_1}\right)^{\frac{1}{4}}, \\
	\end{eqnarray*}
	and the second term is always bounded above by $1-\beta$. We replace the minimizing polynomial in $\mathcal{P}_{m-1}(1)$ by $\widehat Q(x) = \frac{T_{m-1}\left(\frac{2}{\beta} x - 1\right)}{T_{m-1}\left(\frac{2}{\beta} -1\right)}$, where $T_{m-1}(\cdot)$ is the Chebyshev polynomial of the first kind. 
	The Chebyshev polynomials $T_{m-1}(\cdot)$ are bounded by one in magnitude on the interval $[-1,1]$, and this bound is tight at the endpoints. Therefore, our maximum is achieved at $x =0$, and
	$$	\max_{x \in [0,\beta]} | \widehat Q(x)|^{1/2}(1 - x)^{2^{q-2}} =  \left| T_{m-1}\left(\frac{2}{\beta} -1\right) \right|^{-1/2} .$$
	By the definition $T_{m-1}(x) = 1/2 \left( \left(x- \sqrt{x^2-1}\right)^{m-1} + \left(x+ \sqrt{x^2-1}\right)^{m-1}\right)$, $|x| \ge 1$, and the standard inequality $e^{2x} \le \frac{1 +x}{1-x}$, $x \in [0,1]$,
	\begin{eqnarray*}
		T_{m-1}\left(\frac{2}{\beta} -1\right) &\ge& \frac{1}{2} \left(\frac{2}{\beta} -1 + \sqrt{\left(\frac{2}{\beta} -1\right)^2 -1} \right)^{m-1} = \frac{1}{2} \left(\frac{1+\sqrt{1-\beta}}{1-\sqrt{1-\beta}}\right)^{m-1} \\ &\ge& \frac{1}{2} \exp \left\{2 \sqrt{1-\beta} \, (m-1) \right\}.
	\end{eqnarray*}
	In addition,
	$$ \int_{[0,\infty)^n} \left(\frac{\sum_{i=2}^n y_i}{y_1}\right)^{1/4} f_Y(y) \, dy =\frac{\Gamma\left(n/2 -1/4 \right)\Gamma(1/4)}{\Gamma(n/2-1/2) \Gamma(1/2)} \le \frac{\Gamma(1/4) }{2^{1/4} \Gamma(1/2)} n^{1/4},$$
	which gives us
	$$  \E \left[ \left( \frac{\lambda_1 - \lambda_1^{(m)}}{\lambda_1 -\lambda_n}\right)^p \right]^{\tfrac{1}{p}} \le \left[\frac{2^{1/4} \Gamma(1/4)}{\Gamma(1/2)} n^{1/4} \right]^{1/p} e^{- \gamma (m-1) /p} + \gamma^2,$$
	where $\gamma = \sqrt{1-\beta}$. Setting $\gamma = \frac{p}{m-1} \ln \left( n^{1/4p} (m-1)^2 \right)$ (assuming $\gamma<1$, otherwise our bound is already greater than one, and trivially holds), we obtain
	\begin{eqnarray*}
		\E \left[ \left( \frac{\lambda_1 - \lambda_1^{(m)}}{\lambda_1 -\lambda_n}\right)^p \right]^{1/p} &\le& \frac{\left(\frac{2^{1/4} \Gamma(1/4)}{\Gamma(1/2)}\right)^{1/p} +  p^2 \ln^2 (n^{1/4p}(m-1)^2) }{(m-1)^2} \\
		&=&   \frac{\left(\frac{2^{1/4} \Gamma(1/4)}{\Gamma(1/2)}\right)^{1/p} +  \frac{1}{16} \ln^2 \left(n(m-1)^{8p} \right) }{(m-1)^2} \\
		&\le& .068 \, \frac{ \ln^2 \left(n(m-1)^{8p}\right) }{ (m-1)^2},
	\end{eqnarray*}
	for $m \ge 10$, $n \ge 100$. This completes the proof.
	
	
\end{proof}

A similar proof, paired with probabilistic bounds on the quantity $\textstyle \sum_{i=2}^n Y_i/Y_1$, where $Y_1,...,Y_n \sim \chi^2_1$, gives a probabilistic upper estimate. Combining the lower bounds in Lemmas \ref{lm:msquared_lb} and \ref{lm:lowerbound_mix} with the upper bound of Lemma \ref{lm:upper_rel} completes the proof of Theorem \ref{thm:main}.

\section{Distribution Dependent Bounds}

In this section, we consider improved estimates for matrices with certain specific properties. First, we show that if a matrix $A$ has a reasonable number of eigenvalues near $\lambda_1(A)$, then we can produce an error estimate that depends only on the number of iterations $m$. In particular, we suppose that the eigenvalues of our matrix $A$ are such that, once scaled, there are at least $n/(m-1)^\alpha$ eigenvalues in the range $[
\beta,1]$, for a specific value of $\beta$ satisfying $1-\beta = O\left( m^{-2} \ln^2 m  \right)$. For a large number of matrices for which the Lanczos method is used, this assumption holds true. Under this assumption, we prove the following error estimate.

\begin{theorem} \label{thm:eigclust}
	Let $A \in \mathcal{S}^n$, $m \ge 10$, $p\ge 1$, $\alpha > 0$, and $n \ge m(m-1)^{\alpha}$. If
	$$\# \left\{ \lambda_i(A) \, \bigg| \, \frac{\lambda_1(A)-\lambda_i(A)}{\lambda_1(A) - \lambda_n(A)} \le \left(\frac{(2p+\alpha/4) \ln(m-1)}{m-1} \right)^2 \right\} \ge   \frac{n}{(m-1)^\alpha},$$
	then
	$$  \E_{b \sim \mathcal{U}(S^{n-1})} \left[ \left( \frac{\lambda_1(A) - \lambda_1^{(m)}(A,b)}{\lambda_1(A) -\lambda_n(A)}\right)^p \right]^{\frac{1}{p}} \le   .077 \, \frac{(2p + \alpha/4)^2 \ln^2 (m-1)}{(m-1)^2}.$$
	In addition, if
	$$\# \left\{ \lambda_i(A) \, \bigg| \, \frac{\lambda_1(A)-\lambda_i(A)}{\lambda_1(A) - \lambda_n(A)} \le \left(\frac{(\alpha+2) \ln(m-1)}{4(m-1)} \right)^2 \right\} \ge   \frac{n}{(m-1)^\alpha},$$
	then
	$$  \bP_{b \sim \mathcal{U}(S^{n-1})} \left[  \frac{\lambda_1(A) - \lambda_1^{(m)}(A,b)}{\lambda_1(A)-\lambda_n(A)} \le .126 \, \frac{(\alpha+2)^2 \ln^2 (m-1)}{ (m-1)^2} \right] \ge 1 - O(e^{-m}).$$
\end{theorem}

\begin{proof}
	We begin by bounding expected relative error. The main idea is to proceed as in the proof of Lemma \ref{lm:upper_rel}, but take advantage of the number of eigenvalues near $\lambda_1$. For simplicity, let $\lambda_1 = 1$ and $\lambda_n = 0$. We consider eigenvalues in the ranges $[0,2\beta-1)$ and $[2\beta-1,1]$ separately, $1/2< \beta < 1$, and then make use of the lower bound for the number of eigenvalues in $[\beta,1]$. From the proof of Lemma \ref{lm:upper_rel}, we have 
	\begin{eqnarray*}
		\E \left[ \left( \frac{\lambda_1 - \lambda_1^{(m)}}{\lambda_1 -\lambda_n}\right)^p \right]^{\tfrac{1}{p}} 
		&\le& \min_{Q \in \mathcal{P}_{m-1}(1)} \left[ \int_{[0,\infty)^n} \left( \frac{\sum_{i: \lambda_i < 2\beta-1} y_i Q^2(\lambda_i)(1 - \lambda_i)^{2^q}}{ \sum_{i=1}^n y_i Q^2(\lambda_i)} \right)^{\frac{p}{2^q}} f_Y(y) \, dy \right]^{\frac{1}{p}} \\
		&&\qquad + \quad \left[ \int_{[0,\infty)^n} \left( \frac{\sum_{i: \lambda_i \ge 2\beta-1} y_i Q^2(\lambda_i)(1 - \lambda_i)^{2q}}{ \sum_{i=1}^n y_i Q^2(\lambda_i)} \right)^{\frac{p}{2^q}} f_Y(y) \, dy \right]^{\frac{1}{p}},
	\end{eqnarray*}
	where $q \in \mathbb{N}$, $2p < 2^q \le 4p$, $f_Y(y)$ is given by (\ref{eqn:chidensity}), and $\beta \in (1/2,1)$. The second term is at most $2(1-\beta)$, and the integrand of the first term is bounded above by
	\begin{eqnarray*}
		\left( \frac{\sum_{i: \lambda_i < 2\beta-1} y_i Q^2(\lambda_i)(1 - \lambda_i)^{2^q}}{ \sum_{i=1}^n y_i Q^2(\lambda_i)} \right)^{p/2^q} &\le&  \left( \frac{\sum_{i: \lambda_i < 2\beta-1} y_i Q^2(\lambda_i)(1 - \lambda_i)^{2^q}}{ \sum_{i=1}^n y_i Q^2(\lambda_i)} \right)^{1/4} \\
		&\le& \left( \frac{\sum_{i: \lambda_i < 2\beta-1} y_i Q^2(\lambda_i)(1 - \lambda_i)^{2^q}}{ \sum_{i:\lambda_i \ge \beta} y_i Q^2(\lambda_i)} \right)^{1/4}  \\
		&\le&  \frac{ \max_{x \in [0,2\beta-1]} \left|Q(x)\right|^{1/2} (1-x)^{2^{q-2}}}{\min_{x\in [\beta,1]}  \left|Q(x)\right|^{1/2} } \, \left(  \frac{\sum_{i: \lambda_i < 2\beta-1} y_i }{ \sum_{i:\lambda_i \ge \beta} y_i } \right)^{1/4} .
	\end{eqnarray*}
	Let $\sqrt{1-\beta} = \frac{(2p+\alpha/4) \ln(m-1)}{m-1}<1/4$ (if $\beta \le 1/2$, then our estimate is a trivial one, and we are already done). By the condition of the theorem, there are at least $n /(m-1)^\alpha$ eigenvalues in the interval $[\beta,1]$. Therefore,
	\begin{eqnarray*}
		\int_{[0,\infty)^n} \left( \frac{\sum_{i: \lambda_i < 2\beta-1} y_i }{ \sum_{i:\lambda_i \ge \beta} y_i } \right)^{1/4} f_Y(y) \, dy &\le& \E_{\hat Y \sim \chi^2_{n}} \left[\hat Y^{1/4}\right] \E_{\tilde Y \sim \chi^2_{\lceil n/ (m-1)^\alpha \rceil}} \left[\tilde Y^{-1/4}\right] \\
		&=&  \frac{\Gamma(n/2+1/4) \Gamma(\lceil n /(m-1)^\alpha \rceil /2-1/4)}{\Gamma(n/2)\Gamma(\lceil n /(m-1)^\alpha \rceil /2)} \\
		&\le& \frac{\Gamma(m(m-1)^\alpha/2+1/4) \Gamma(m /2-1/4)}{\Gamma(m(m-1)^\alpha/2)\Gamma(m /2)} \\
		&\le& 1.04  \, (m-1)^{\alpha/4}
	\end{eqnarray*}
	for $n \ge m(m-1)^{\alpha}$ and $m \ge 10$. Replacing the minimizing polynomial by $\widehat Q(x) = \frac{T_{m-1}\left( \frac{2x}{2\beta-1} - 1 \right)}{T_{m-1}\left( \frac{2}{2\beta-1} -1\right)}$,
	we obtain
	$$ \frac{ \max_{x \in [0,2\beta-1]} \left|\widehat Q(x)\right|^{\tfrac{1}{2}} (1-x)^{2^{q-2}}}{\min_{x\in [\beta,1]} \left| \widehat Q(x)\right|^{\tfrac{1}{2}} }  = \frac{1}{T_{m-1}^{1/2}\left( \frac{2\beta}{2\beta-1} - 1 \right)} \le \frac{1}{T_{m-1}^{1/2}\left( \frac{2}{\beta} - 1 \right)} \le \sqrt{2} e^{-\sqrt{1-\beta}(m-1)}.$$
	Combining our estimates results in the bound
	\begin{eqnarray*}
		\E \left[ \left( \frac{\lambda_1 - \lambda_1^{(m)}}{\lambda_1 -\lambda_n}\right)^p \right]^{\tfrac{1}{p}} 
		&\le& \left(1.04 \sqrt{2} (m-1)^{\alpha/4} \right)^{1/p} e^{-\sqrt{1-\beta}(m-1)/p} + 2(1-\beta)\\
		&=& \frac{(1.04\sqrt{2})^{1/p}}{(m-1)^2} + \frac{(2p + \alpha/4)^2 \ln^2 (m-1)}{(m-1)^2} \\
		&\le& .077 \, \frac{(2p + \alpha/4)^2 \ln^2 (m-1)}{(m-1)^2}
	\end{eqnarray*}
	for $m \ge 10$, $p \ge 1$, and $\alpha > 0$. This completes the bound for expected relative error. We now produce a probabilistic bound for relative error. Let $\sqrt{1-\beta} = \frac{(\alpha+2) \ln(m-1)}{4(m-1)}$.
	We have
	\begin{eqnarray*}
		\frac{ \lambda_1 - \lambda_1^{(m)}}{\lambda_1-\lambda_n} &=& \min_{Q \in \mathcal{P}_{m-1}(1)} \frac{\sum_{i=2}^n Y_i Q^2(\lambda_i) (1- \lambda_i)}{\sum_{i=1}^n Y_i Q^2(\lambda_i) } \\
		&\le& \min_{Q \in \mathcal{P}_{m-1}(1)} \frac{\max_{x \in [0 , 2\beta-1]} Q^2(x) (1-x)}{\min_{x\in [\beta,1]}  Q^2(x)  } \frac{\sum_{i: \lambda_i < 2\beta-1} Y_i}{\sum_{i: \lambda_i \ge \beta} Y_i} + 2(1-\beta) \\
		&\le& T_{m-1}^{-2} \left( \frac{2}{\beta} -1 \right)  \frac{\sum_{i: \lambda_i < 2\beta-1} Y_i}{\sum_{i: \lambda_i \ge \beta} Y_i} + 2(1-\beta) \\
		&\le& 4 \exp \{ -4 \sqrt{1-\beta} (m-1)\}  \frac{\sum_{i: \lambda_i < 2\beta-1} Y_i}{\sum_{i: \lambda_i \ge \beta} Y_i} + 2(1-\beta).
	\end{eqnarray*}
	By Proposition \ref{prop:cher},
	\begin{eqnarray*}
		\bP\left[\frac{\sum_{i: \lambda_i < 2\beta-1} Y_i}{\sum_{i: \lambda_i \ge \beta} Y_i} \ge 4 (m-1)^{\alpha} \right] &\le& \bP\left[ \sum_{i: \lambda_i < 2\beta-1} Y_i \ge 2 n \right] + \bP\left[ \sum_{i: \lambda_i \ge \beta} Y_i \le  \frac{n}{2(m-1)^{\alpha}}\right] \\
		&\le& \left(2/e \right)^{n/2} + \left( \sqrt{e}/2 \right)^{n/2(m-1)^{\alpha}} = O(e^{-m}).
	\end{eqnarray*}
	Then, with probability $1-O(e^{-m})$,
	$$\frac{ \lambda_1 - \lambda_1^{(m)}}{\lambda_1-\lambda_n} \,\le \, 16 (m-1)^\alpha e^{-4 \sqrt{1-\beta} (m-1)} +2(1-\beta) \,=\, \frac{16}{(m-1)^2} + \frac{(\alpha+2)^2 \ln^2 (m-1)}{8 (m-1)^2}.$$
	The $16/(m-1)^2$ term is dominated by the log term as $m$ increases, and, therefore, with probability $1-O(e^{-m})$,
	$$\frac{ \lambda_1 - \lambda_1^{(m)}}{\lambda_1-\lambda_n} \,\le\, .126 \, \frac{(\alpha+2)^2 \ln^2 (m-1)}{ (m-1)^2}.$$
	This completes the proof.
\end{proof}

The above theorem shows that, for matrices whose distribution of eigenvalues is independent of $n$, we can obtain dimension-free estimates. For example, the above theorem holds for the matrix from Example \ref{ex:tridiag}, for $\alpha = 2$.

When a matrix has eigenvalues known to converge to a limiting distribution as dimension increases, or a random matrix $X_n$ exhibits suitable convergence of its empirical spectral distribution $\textstyle L_{X_n} := 1/n \sum_{i=1}^n \delta_{\lambda_i(X_n)}$, improved estimates can be obtained by simply estimating the corresponding integral polynomial minimization problem. However, to do so, we first require a law of large numbers for weighted sums of independent identically distributed (i.i.d.) random variables. We recall the following result.

\begin{proposition}{(\cite{cuzick1995strong})}\label{prop:lawlarge}
	Let $a_1,a_2,... \in [a,b]$ and $X_1,X_2,...$ be i.i.d. random variables, with $\mathbb{E}[X_1] = 0$ and $\mathbb{E}[X_1^2] < \infty$. Then
	$\textstyle{\frac{1}{n} \sum_{i=1}^n a_i X_i \rightarrow 0}$ almost surely.
\end{proposition}

We present the following theorem regarding the error of random matrices that exhibit suitable convergence.

\begin{theorem} \label{thm:conv}
	Let $X_n \in \mathcal{S}^n$, $\Lambda(X_n) \in[a,b]$, $n = 1,2,...$ be a sequence of random matrices, such that $L_{X_n}$ converges in probability to $\sigma(x) \, dx$ in $L_2([a,b])$, where $\sigma(x) \in C([a,b])$, $a,b \in \text{supp}(\sigma)$.  Then, for all $m \in \mathbb{N}$ and $\epsilon >0$,
	$$ \lim_{n \rightarrow \infty} \mathbb{P} \bigg( \bigg| \frac{\lambda_1(X_n) - \lambda^{(m)}_1(X_n,b)}{\lambda_1(X_n) - \lambda_n(X_n)} - \frac{b-\xi(m)}{b-a} \bigg| > \epsilon \bigg) = 0,$$
	where $\xi(m)$ is the largest zero of the $m^{th}$ degree orthogonal polynomial of $\sigma(x)$ in the interval $[a,b]$.
\end{theorem}

\begin{proof}
	The main idea of the proof is to use Proposition \ref{prop:lawlarge} to control the behavior of $Y$, and the convergence of $L_{X_n}$ to $\sigma(x) \, dx$ to show convergence to our integral minimization problem. We first write our polynomial $P \in \mathcal{P}_{m-1}$ as $\textstyle P(x)  = \sum_{j=0}^{m-1} \alpha_j x^j$ and our unnormalized error as
	\begin{eqnarray*}
		\lambda_1(X_n) - \lambda^{(m)}_1(X_n,b) & = & \min_{\substack{P \in \mathcal{P}_{m-1} \\ P \ne 0}} \frac{\sum_{i=2}^n Y_i P^2(\lambda_i)(\lambda_1-\lambda_i) }{\sum_{i=1}^n Y_i P^2(\lambda_i)} \\ & = & \min_{\substack{\alpha \in \mathbb{R}^{m} \\ \alpha \ne 0}} \frac{\sum_{j_1,j_2 = 0}^{m-1} \alpha_{j_1} \alpha_{j_2} \sum_{i=2}^n Y_i \lambda_i^{j_1+j_2}(\lambda_1-\lambda_i) }{\sum_{j_1,j_2 = 0}^{m-1} \alpha_{j_1} \alpha_{j_2} \sum_{i=1}^n Y_i \lambda_i^{j_1+j_2} },
	\end{eqnarray*}
	where $Y_1,...,Y_n$ are i.i.d. chi-square random variables with one degree of freedom each. The functions $x^j$, $j = 0,...,2m-2$, are bounded on $[a,b]$, and so, by Proposition \ref{prop:lawlarge}, for any $\epsilon_1,\epsilon_2 >0$,
	$$\bigg| \frac{1}{n} \sum_{i=2}^n Y_i \lambda_i^{j}(\lambda_1-\lambda_i) - \frac{1}{n} \sum_{i=2}^n \lambda_i^{j}(\lambda_1-\lambda_i)  \bigg|< \epsilon_1, \qquad \qquad j = 0,...,2m-2,$$
	and
	$$\bigg| \frac{1}{n} \sum_{i=1}^n Y_i \lambda_i^{j}- \frac{1}{n} \sum_{i=1}^n \lambda_i^{j}  \bigg|< \epsilon_2, \qquad \qquad j = 0,...,2m-2,$$
	with probability $1-o(1)$. $L_{X_n}$ converges in probability to $\sigma(x) \, dx$, and so, for any $\epsilon_3,\epsilon_4>0$,
	$$  \bigg| \frac{1}{n} \sum_{i=2}^n \lambda_i^{j}(\lambda_1-\lambda_i)  - \int_{a}^b x^j (b-x) \sigma(x) \, dx \bigg| < \epsilon_3, \qquad  \qquad j = 0,...,2m-2 ,$$
	and
	$$  \bigg| \frac{1}{n} \sum_{i=1}^n \lambda_i^{j}  - \int_{a}^b x^j \sigma(x) \, dx \bigg| < \epsilon_4, \qquad \qquad  j = 0,...,2m-2 ,$$
	with probability $1-o(1)$. This implies that
	$$ \lambda_1(X_n) - \lambda^{(m)}_1(X_n,b) = \min_{\substack{\alpha \in \mathbb{R}^{m} \\ \alpha \ne 0} } \frac{\sum_{j_1,j_2 = 0}^{m-1} \alpha_{j_1} \alpha_{j_2} \left[ \int_{a}^b x^{j_1+j_2} (b-x) \sigma(x) \, dx + \hat E(j_1,j_2) \right] }{\sum_{j_1,j_2 = 0}^{m-1} \alpha_{j_1} \alpha_{j_2} \left[\int_{a}^b x^{j_1+j_2}  \sigma(x) \, dx  +  E(j_1,j_2) \right] }, $$
	where $|\hat E(j_1,j_2)|< \epsilon_1 + \epsilon_3$ and $| E(j_1,j_2)| < \epsilon_2 + \epsilon_4$, $j_1,j_2 = 0,...,m-1$, with probability $1-o(1)$.
	
	The minimization problem 
	$$	\min_{\substack{P \in \mathcal{P}_{m-1} \\ P \ne 0}} \frac{\int_{a}^b P^2(x) \left( \frac{b-x}{b-a} \right) \sigma(x) \, dx}{\int_{a}^b P^2(x)  \sigma(x) \, dx}$$
	corresponds to a generalized Rayleigh quotient $\frac{\alpha^T A \alpha}{\alpha^T B \alpha}$, where $A,B \in \mathcal{S}^m_{++}$ and $\lambda_{max}(A)$, $\lambda_{max}(B)$, and $\lambda_{min}(B)$ are all constants independent of $n$, and $\alpha = (\alpha_0,...,\alpha_{m-1})^T$. By choosing $\epsilon_1,\epsilon_2,\epsilon_3,\epsilon_4$ sufficiently small,
	\begin{eqnarray*}
		\left|\frac{\alpha^T (A + \hat E) \alpha}{\alpha^T (B+E) \alpha} - \frac{\alpha^T A \alpha}{\alpha^T B \alpha} \right|  & = & \left| \frac{(\alpha^T \hat E \alpha) \alpha^T B \alpha - (\alpha^T  E \alpha) \alpha^T A \alpha }{(\alpha^T B \alpha + \alpha^T  E \alpha) \alpha^T B \alpha} \right| \\
		& \le & \frac{(\epsilon_1 +\epsilon_3)m \lambda_{max}(B) +(\epsilon_2 +\epsilon_4)m \lambda_{max}(A)  }{(\lambda_{min}(B)-(\epsilon_2 + \epsilon_4)m) \lambda_{min}(B)}\le \epsilon
	\end{eqnarray*}
	for all $\alpha \in \mathbb{R}^{m}$ with probability $1-o(1)$. Applying Proposition \ref{prop:intmax} to the above integral minimization problem completes the proof.
\end{proof}

The above theorem is a powerful tool for explicitly computing the error in the Lanczos method for certain types of matrices, as the computation of the extremal eigenvalue of an $m \times m$ matrix (corresponding to the largest zero) is a nominal computation compared to one application of an $n$ dimensional matrix. In Section 6, we will see that this convergence occurs quickly in practice. In addition, the above result provides strong evidence that the inverse quadratic dependence on $m$ in the error estimates throughout this paper is not so much a worst case estimate, but actually indicative of error rates in practice. For instance, if the eigenvalues of a matrix are sampled from a distribution bounded above and below by some multiple of a Jacobi weight function, then, by equation (iv) Proposition \ref{prop:jac2} and Theorem \ref{thm:conv}, it immediately follows that the error is of order $m^{-2}$. Of course, we note that Theorem \ref{thm:conv} is equally applicable for estimating $\lambda_n$.

\section{Estimates for Arbitrary Eigenvalues and Condition Number}

Up to this point, we have concerned ourselves almost exclusively with the extremal eigenvalues $\lambda_1$ and $\lambda_n$ of a matrix. In this section, we extend the techniques of this paper to arbitrary eigenvalues, and also obtain bounds for the condition number of a positive definite matrix. The results of this section provide the first uniform error estimates for arbitrary eigenvalues and the first lower bounds for the relative error with respect to the condition number. Lower bounds for arbitrary eigenvalues follow relatively quickly from our previous work. However, our proof technique for upper bounds requires some mild assumptions regarding the eigenvalue gaps of the matrix. We begin with asymptotic lower bounds for an arbitrary eigenvalue, and present the following corollary of Theorem \ref{thm:main}. 

\begin{corollary}\label{cor:eig_lb}
	$$\max_{A \in \mathcal{S}^n}  \bP_{b \sim \mathcal{U}(S^{n-1})} \left[  \frac{\lambda_i(A) - \lambda_i^{(m)}(A,b)}{\lambda_1(A)-\lambda_n(A)} \ge   1-o(1) \right]  \ge 1 - o(1/n)$$
	for $m = o(\ln n)$,
	$$\max_{A \in \mathcal{S}^n}  \bP_{b \sim \mathcal{U}(S^{n-1})} \left[  \frac{\lambda_i(A) - \lambda_i^{(m)}(A,b)}{\lambda_1(A)-\lambda_n(A)} \ge   \frac{.015 \ln^2 n}{ m^2 \ln^2 \ln n} \right]  \ge 1 - o(1/n) $$
	for $m = \Theta(\ln n)$, and
	$$ \max_{A \in \mathcal{S}^n}  \bP_{b \sim \mathcal{U}(S^{n-1})} \left[  \frac{\lambda_i(A) - \lambda_i^{(m)}(A,b)}{\lambda_1(A)-\lambda_n(A)} \ge \frac{1.08}{m^2} \right]  \ge 1 - o(1/n)$$
	for $m = o\left(n^{1/2} \ln^{-1/2} n\right)$ and $\omega(1)$.
\end{corollary}

\begin{proof}
	Let $\varphi_1,...,\varphi_n$ be the eigenvectors corresponding to $\lambda_1(A) \ge ... \ge \lambda_n(A)$, and $\hat b = b-\sum_{j=1}^{i-1} (\varphi_j^T b) b$. By the inequalities
	$$	\lambda^{(m)}_i (A,b) \le \max_{\substack{x \in \mathcal{K}_m(A,b) \backslash 0  \\ x^T \varphi_j = 0, \\ j = 1,...,i-1}}  \, \frac{x^T A x}{x^T x} \le  \max_{\substack{x \in \mathcal{K}_m\left(A,\hat b\right) \backslash 0}}  \, \frac{x^T A x}{x^T x} = \lambda_1^{(m)}\big(A,\hat b \big),$$
	the relative error 
	$$\frac{\lambda_i(A) - \lambda_i^{(m)}(A,b)}{\lambda_1(A) - \lambda_n(A)} \ge \frac{\lambda_i(A) - \lambda_1^{(m)}(A,\hat b )}{\lambda_1(A) - \lambda_n(A)}.$$
	The right-hand side corresponds to an extremal eigenvalue problem of dimension $n-i+1$. Setting the largest eigenvalue of $A$ to have multiplicity $i$, and defining the eigenvalues $\lambda_i,...,\lambda_n$ based on the eigenvalues (corresponding to dimension $n-i+1$) used in the proofs of Lemmas \ref{lm:msquared_lb} and \ref{lm:lowerbound_mix} completes the proof.
\end{proof}

Next, we provide an upper bound for the relative error in approximating $\lambda_i$ under the assumption of non-zero gaps between eigenvalues $\lambda_1,...,\lambda_i$.

\begin{theorem}\label{thm:arb_eig}
	Let $n \ge 100$, $m \ge 9+i$, $p \ge 1$, and $A \in \mathcal{S}^n$. Then
	$$ \E_{b \sim \mathcal{U}(S^{n-1})} \left[ \left( \frac{\lambda_i(A) - \lambda_i^{(m)}(A,b)}{\lambda_i(A)-\lambda_n(A)}\right)^p \right]^{1/p}  \le .068 \, \frac{ \ln^2 \left( \delta^{-2(i-1)} n(m-i)^{8p}\right) }{(m-i)^2}$$
	and
	$$ \bP_{b \sim \mathcal{U}(S^{n-1})} \left[  \frac{\lambda_i(A) - \lambda_i^{(m)}(A,b)}{\lambda_i(A)-\lambda_n(A)} \le .571 \; \frac{\ln^2\left( \delta^{-2(i-1)/3} n (m-i)^{2/3} \right)}{(m-i)^2} \right]  \ge 1 - o(1/n),$$
	where
	$$ \delta  = \frac{1}{2} \min_{k = 2,...,i} \frac{\lambda_{k-1}(A) - \lambda_{k}(A)}{\lambda_1(A) -\lambda_n(A)}.$$
\end{theorem}

\begin{proof}
	As in previous cases, it suffices to prove the theorem for matrices $A$ with $\lambda_i(A) =1$ and $\lambda_n(A) = 0$ (if $\lambda_i = \lambda_n$, we are done). Similar to the polynomial representation of $\lambda_1^{(m)}(A,b)$, the Ritz value $\lambda_i^{(m)}(A,b)$ corresponds to finding the polynomial in $\mathcal{P}_{m-1}$ with zeros $\lambda_k^{(m)}$, $k = 1,...,i-1$, that maximizes the corresponding Rayleigh quotient. For the sake of brevity, let $\phi_i(x) = \prod_{k=1}^{i-1} (\lambda_k^{(m)} - x)^2$.  Then $\lambda_i^{(m)}(A,b)$ can be written as
	$$ \lambda_i^{(m)}(A,b) = \max_{P \in \mathcal{P}_{m-i}} \frac{\sum_{j=1}^n b_j^2 P^2(\lambda_j) \phi_i(\lambda_j) \lambda_j   }{\sum_{j=1}^n b_j^2 P^2(\lambda_j) \phi_i(\lambda_j) },$$
	and, therefore, the error is bounded above by
	$$	 \lambda_i(A) - \lambda_i^{(m)}(A,b) \le
	\max_{P \in \mathcal{P}_{m-i}} \frac{\sum_{j=i+1}^n b_j^2 P^2(\lambda_j)  \phi_i(\lambda_j) (1 -\lambda_j) }{\sum_{j=1}^n b_j^2 P^2(\lambda_j) \phi_i(\lambda_j) } .$$
	The main idea of the proof is very similar to that of Lemma \ref{lm:upper_rel}, paired with a pigeonhole principle. The intervals $[\lambda_j(A) - \delta \lambda_1 , \lambda_j(A) +\delta\lambda_1]$, $j =1,...,i$, are disjoint, and so there exists some index $j^*$ such that the corresponding interval does not contain any of the Ritz values $\lambda_k^{(m)}(A,b)$, $k = 1,...,i-1$. We begin by bounding expected relative error. As in the proof of Lemma \ref{lm:upper_rel}, by integrating over chi-square random variables and using Cauchy-Schwarz, we have
	\begin{eqnarray*}
		\E \left[ \big( \lambda_i - \lambda_i^{(m)}\big)^p \right]^{\tfrac{1}{p}} &\le& \min_{P \in \mathcal{P}_{m-i}} \left[ \int_{[0,\infty)^n} \left( \frac{\sum_{j=i+1}^n y_j P^2(\lambda_j)  \phi_i(\lambda_j) (1 -\lambda_j)^{2^q}}{\sum_{j=1}^n y_j P^2(\lambda_j)\phi_i(\lambda_j)  } \right)^{\frac{p}{2^q}} f_Y(y) \, dy \right]^{\frac{1}{p}} \\
		&\le& \min_{P \in \mathcal{P}_{m-i}} \left[ \int_{[0,\infty)^n} \left( \frac{\sum_{j: \lambda_j < \beta} y_j P^2(\lambda_j) \phi_i(\lambda_j) (1 -\lambda_j)^{2^q} }{\sum_{j=1}^n y_j P^2(\lambda_j)\phi_i(\lambda_j)  } \right)^{\frac{p}{2^q}} f_Y(y) \, dy \right]^{\frac{1}{p}}  \\
		&&+ \quad  \left[ \int_{[0,\infty)^n} \left( \frac{\sum_{j: \lambda_j \in [\beta,1]} y_j P^2(\lambda_j)  \phi_i(\lambda_j) (1 -\lambda_j)^{2^q} }{\sum_{j=1}^n y_j P^2(\lambda_j)\phi_i(\lambda_j)  } \right)^{\frac{p}{2^q}} f_Y(y) \, dy \right]^{\frac{1}{p}} \\
	\end{eqnarray*}
	for $q \in \mathbb{N}$, $2p < 2^q \le 4 p$, and any $ \beta \in (0,1)$. The second term on the right-hand side is bounded above by $1 - \beta$, and the integrand of the first term is bounded above by	
	\begin{eqnarray*}
		\left( \frac{\sum_{j: \lambda_j < \beta} y_j P^2(\lambda_j) \phi_i(\lambda_j) (1 -\lambda_j)^{2^q}  }{\sum_{j=1}^n y_j P^2(\lambda_j)\phi_i(\lambda_j)  } \right)^{\frac{p}{2^q}}  &\le&  \left( \frac{\sum_{j: \lambda_j < \beta} y_j P^2(\lambda_j)  \phi_i(\lambda_j) (1 -\lambda_j)^{2^q} }{\sum_{j=1}^n y_j P^2(\lambda_j)\phi_i(\lambda_j)  } \right)^{\frac{1}{4}}  \\
		&\le& \max_{x \in [0,\beta]} \frac{|P(x)|^{1/2} \phi^{1/4}_i(x) (1 - x)^{2^{q-2}}}{ |P(\lambda_{j^*})|^{1/2} \phi^{1/4}_i(\lambda_{j^*}) } \left(\frac{\sum_{i:\lambda_i<\beta} y_i}{y_{j^*}}\right)^{\frac{1}{4}}. \\
	\end{eqnarray*}
	By replacing the minimizing polynomial in $\mathcal{P}_{m-i}$ by $ T_{m-i}\left(\frac{2}{\beta} x - 1\right)$, the maximum is achieved at $x =0$, and, by the monotonicity of $T_{m-i}$ on $[1,\infty)$,
	$$T_{m-i}\left(\frac{2}{\beta} \lambda_{j^*} -1\right) 
	\ge	T_{m-i}\left(\frac{2}{\beta} -1\right) 
	\ge \frac{1}{2} \exp \left\{2 \sqrt{1-\beta} \, (m-i) \right\}.$$
	In addition,
	$$ \frac{\phi^{1/4}_i(0) }{\phi^{1/4}_i(\lambda_{j^*})} = \prod_{k=1}^{i-1}  \left| \frac{\lambda_k^{(m)}}{\lambda_k^{(m)} - \lambda_{j^*}} \right|^{1/2}  \le \delta^{-(i-1)/2}.$$
	We can now bound the $p$-norm by
	$$  \E \left[ \left( \lambda_i - \lambda_i^{(m)}\right)^p \right]^{\tfrac{1}{p}} \le  \left[\frac{2^{1/4} \Gamma(1/4)}{\Gamma(1/2)} \frac{n^{1/4}}{\delta^{(i-1)/2}  } \right]^{1/p} e^{- \gamma (m-i) /p} + \gamma^2,$$
	where $\gamma = \sqrt{1-\beta}$. Setting $\gamma = \frac{p}{m-i} \ln \left( \delta^{-(i-1)/2p}n^{1/4p} (m-i)^2 \right)$ (assuming $\gamma<1$, otherwise our bound is already greater than one, and trivially holds), we obtain
	\begin{eqnarray*}
		\E \left[ \left( \lambda_i - \lambda_i^{(m)}\right)^p \right]^{\tfrac{1}{p}}
		&\le&   \frac{\left(\frac{2^{1/4} \Gamma(1/4)}{\Gamma(1/2)}\right)^{1/p} +  \frac{1}{16} \ln^2 \left(\delta^{-2(i-1)}n(m-i)^{8p} \right) }{(m-i)^2} \\
		&\le& .068 \, \frac{ \ln^2 \left(\delta^{-2(i-1)} n(m-i)^{8p}\right) }{ (m-i)^2},
	\end{eqnarray*}
	for $m \ge 9+i$, $n \ge 100$. This completes the proof of the expected error estimate. We now focus on the probabilistic estimate. We have
	\begin{eqnarray*}
		\lambda_i(A) - \lambda_i^{(m)}(A,b) &\le& \min_{P \in \mathcal{P}_{m-i}} \frac{\sum_{j=i+1}^n Y_j P^2(\lambda_j) \phi_i(\lambda_j) (1- \lambda_j)}{\sum_{j=1}^n Y_j P^2(\lambda_j) \phi_i(\lambda_j)} \\
		&\le& \min_{P \in \mathcal{P}_{m-i}} \max_{x \in [0 , \beta]} \frac{P^2(x)  \phi_i(x) (1-x)}{P^2(\lambda_{j^*}) \phi_i(\lambda_{j^*})} \frac{\sum_{j:\lambda_j<\beta} Y_j}{Y_{j^*}} + (1-\beta) \\
		&\le& \delta^{-2(i-1)} \, T_{m-i}^{-2} \left( \frac{2}{\beta} -1 \right) \frac{\sum_{j:\lambda_j<\beta} Y_j}{Y_{j^*}} + (1-\beta) \\
		&\le& 4 \, \delta^{-2(i-1)} \exp \{ -4 \sqrt{1-\beta} (m-i)\} \frac{\sum_{j:\lambda_j<\beta} Y_j}{Y_{j^*}} + (1-\beta).
	\end{eqnarray*}
	By Proposition \ref{prop:cher},
	\begin{eqnarray*}
		\bP\left[\frac{\sum_{j:\lambda_j<\beta} Y_j}{Y_{j^*}} \ge n^{3.02} \right] &\le& \bP\left[ Y_{j^*} \le n^{-2.01} \right] + \bP\left[ \sum_{j\ne j^*} Y_j \ge n^{1.01} \right] \\
		&\le& \left(e/n^{2.01}\right)^{1/2} + \left( n^{.01}e^{1-n^{.01}}\right)^{(n-1)/2} = o(1/n).
	\end{eqnarray*}
	Let $\sqrt{1-\beta} = \frac{ \ln \left( \delta^{-2(i-1)} n^{3.02}  (m-i)^{2}\right)}{4 (m-i)}$. Then, with probability $1-o(1/n)$,
	$$ \lambda_i(A) - \lambda_i^{(m)}(A) \le
	\frac{4}{(m-i)^2} + \frac{\ln^2\left( \delta^{-2(i-1)} n^{3.02}  (m-i)^{2}\right)}{16(m-i)^2}.$$
	The $4/(m-i)^2$ term is dominated by the log term as $n$ increases, and, therefore, with probability $1-o(1/n)$,
	$$\lambda_i(A) - \lambda_i^{(m)}(A) \le .571 \; \frac{\ln^2\left( \delta^{-2(i-1)/3} n (m-i)^{2/3} \right)}{(m-i)^2}.$$
	This completes the proof.
\end{proof}

For typical matrices with no repeated eigenvalues, $\delta$ is usually a very low degree polynomial in $n$, and, for $i$ constant, the estimates for $\lambda_i$ are not much worse than that of $\lambda_1$. In addition, it seems likely that, given any matrix $A$, a small random perturbation of $A$ before the application of the Lanczos method will satisfy the condition of the theorem with high probability, and change the eigenvalues by a negligible amount. Of course, the bounds from Theorem \ref{thm:arb_eig} for maximal eigenvalues $\lambda_i$ also apply to minimal eigenvalues $\lambda_{n-i}$.

Next, we make use of Theorem \ref{thm:main} to produce error estimates for the condition number of a symmetric positive definite matrix. Let $\kappa(A) = \lambda_1(A) / \lambda_n(A)$ denote the condition number of a matrix $A \in \mathcal{S}^n_{++}$ and $\kappa^{(m)}(A,b) = \lambda_1^{(m)}(A,b) / \lambda_m^{(m)}(A,b)$ denote the condition number of the tridiagonal matrix $T_m \in \mathcal{S}^m_{++}$ resulting from $m$ iterations of the Lanczos method applied to a matrix $A \in \mathcal{S}^{n}_{++}$ with initial vector $b$. Their difference can be written as
\begin{eqnarray*}
	\kappa(A) - \kappa^{(m)}(A,b) &=& \frac{\lambda_1}{\lambda_n} - \frac{\lambda_1^{(m)} }{\lambda_m^{(m)}} \,=\, \frac{\lambda_1\lambda_m^{(m)} - \lambda_n \lambda_1^{(m)} }{ \lambda_n \lambda_m^{(m)}} = \frac{ \lambda_m^{(m)} \left( \lambda_1 - \lambda_1^{(m)} \right) + \lambda_1^{(m)} \left(\lambda_m^{(m)} - \lambda_n \right) }{ \lambda_n \lambda_m^{(m)}} \\
	&=& \left(\kappa(A) -1 \right) \left[ \frac{\lambda_1 - \lambda_1^{(m)}}{\lambda_1-\lambda_n} +\kappa^{(m)}(A,b) \frac{\lambda_m^{(m)} - \lambda_n}{\lambda_1-\lambda_n} \right],
\end{eqnarray*}
which leads to the bounds
$$  \frac{\kappa(A) -\kappa^{(m)}(A,b)}{\kappa(A)} \le  \frac{\lambda_1 - \lambda_1^{(m)}}{\lambda_1-\lambda_n} +\kappa(A) \frac{\lambda_m^{(m)} - \lambda_n}{\lambda_1-\lambda_n}$$
and
$$\frac{\kappa(A) -\kappa^{(m)}(A,b)}{\kappa^{(m)}(A,b)} \ge  (\kappa(A) -1) \frac{\lambda_m^{(m)} - \lambda_n}{\lambda_1-\lambda_n}.$$
Using Theorem \ref{thm:main} and Minkowski's inequality, we have the following corollary.

\begin{corollary}\label{cor:cond_exp}
	$$\sup_{\substack{A \in \mathcal{S}^n_{++} \\ \kappa(A) =\bar \kappa}} \bP_{b \sim \mathcal{U}(S^{n-1})} \left[  \frac{\kappa(A) -\kappa^{(m)}(A,b)}{\kappa^{(m)}(A,b)} \ge (1-o(1)) (\bar \kappa -1)\right]  \ge 1 - o(1/n)$$
	for $m = o(\ln n)$,
	$$\sup_{\substack{A \in \mathcal{S}^n_{++} \\ \kappa(A) =\bar \kappa}}   \bP_{b \sim \mathcal{U}(S^{n-1})} \left[  \frac{\kappa(A) -\kappa^{(m)}(A,b)}{\kappa^{(m)}(A,b)} \ge  .015  \, \frac{(\bar \kappa -1) \ln^2 n}{ m^2 \ln^2 \ln n} \right]  \ge 1 - o(1/n)$$
	for $m = \Theta(\ln n)$,
	$$ \sup_{\substack{A \in \mathcal{S}^n_{++} \\ \kappa(A) =\bar \kappa}}   \bP_{b \sim \mathcal{U}(S^{n-1})} \left[  \frac{\kappa(A) -\kappa^{(m)}(A,b)}{\kappa^{(m)}(A,b)}\ge 1.08 \, \frac{\bar \kappa -1}{m^2} \right]  \ge 1 - o(1/n)$$
	for $m = o\left(n^{1/2} \ln^{-1/2} n\right)$ and $\omega(1)$, and
	$$  \sup_{\substack{A \in \mathcal{S}^n_{++} \\ \kappa(A) =\bar \kappa}} \E_{b \sim \mathcal{U}(S^{n-1})} \left[ \left( \frac{\kappa(A) -\kappa^{(m)}(A,b)}{\kappa(A)} \right)^p \right]^{1/p} \le .068\, (\bar \kappa +1) \,  \frac{ \ln^2 \left(n(m-1)^{8p}\right) }{ (m-1)^2}$$
	for $n\ge 100$, $m \ge 10$, $p \ge 1$.
\end{corollary}

%
%

As previously mentioned, it is not possible to produce uniform bounds for the relative error of $\kappa^{(m)}(A,b)$, and so some dependence on $\kappa(A)$ is necessary.

\section{Experimental Results}

In this section, we present a number of experimental results that illustrate the error of the Lanczos method in practice. We consider:
\begin{itemize}
	\item eigenvalues of 1D Laplacian with Dirichlet boundary conditions $\Lambda_{lap}=  \{2 + 2 \cos( i\pi/(n+1))\}_{i=1}^n$,
	\item a uniform partition of $[0,1]$, $\Lambda_{unif} = \{(n-i)/(n-1)\}_{i=1}^n$,
	\item eigenvalues from the semi-circle distribution, $\Lambda_{semi} = \{\lambda_i\}_{i=1}^n$, where \\
	$1/2 +  \big(\lambda_i \sqrt{1 - \lambda_i^2} + \arcsin \lambda_i \big)/\pi = (n-i)/(n-1)$, $i = 1,...,n$,
	\item eigenvalues corresponding to Lemma \ref{lm:lowerbound_mix}, $\Lambda_{log}= \{ 1 - \left[(n+1-i)/n\right]^{\frac{\ln \ln n}{\ln n}}\}_{i=1}^n$.
\end{itemize}

\begin{figure}
	\begin{center}
		\subfloat[Plot, $\Lambda_{lap}$]{\includegraphics*[width=2.5 in, height = 1.7 in]{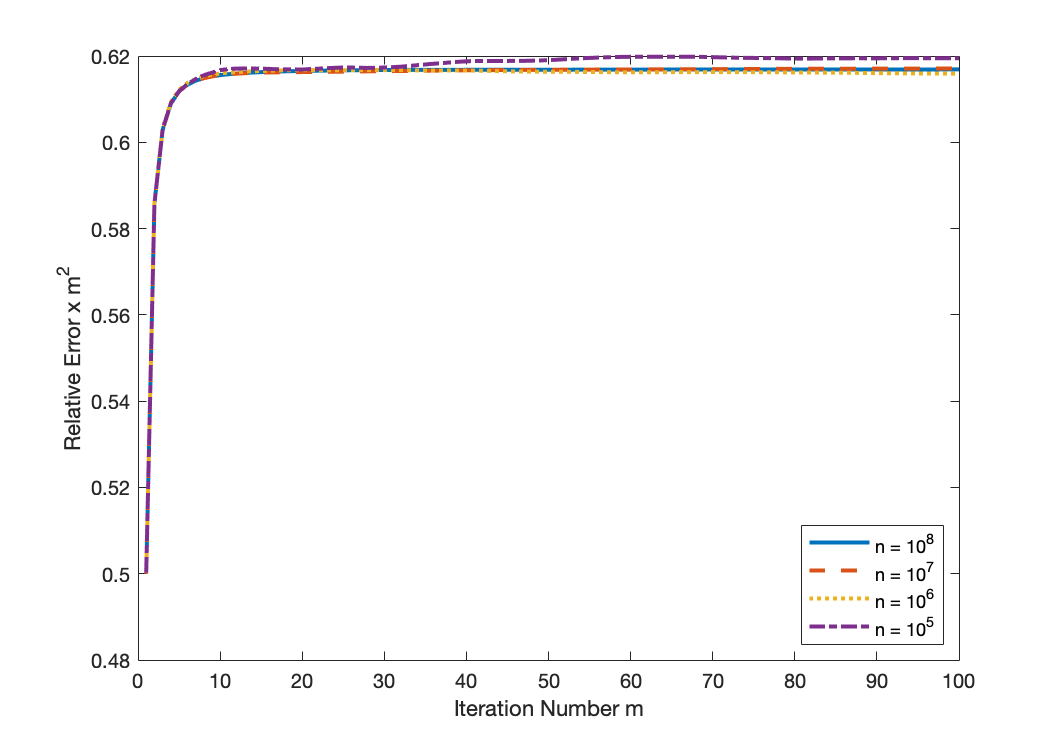}} \subfloat[Box Plot, $\Lambda_{lap}$]{\includegraphics*[width= 2.5 in, height = 1.7 in]{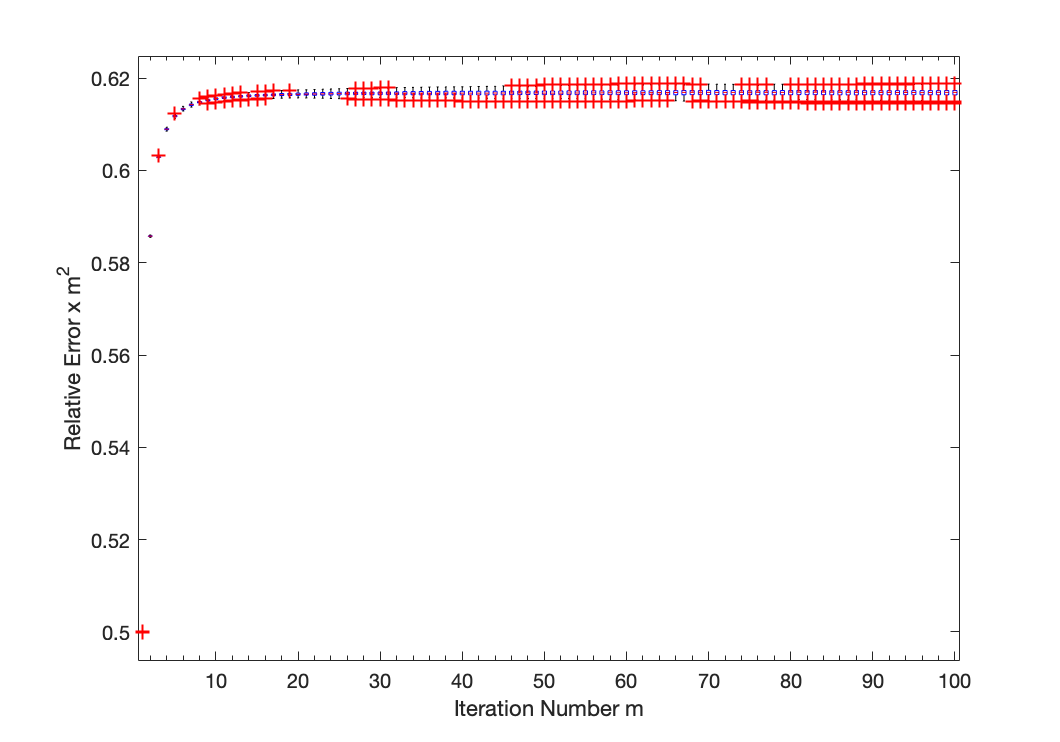}}\\
		\subfloat[Plot, $\Lambda_{unif}$]{\includegraphics*[width=2.5 in, height = 1.7 in]{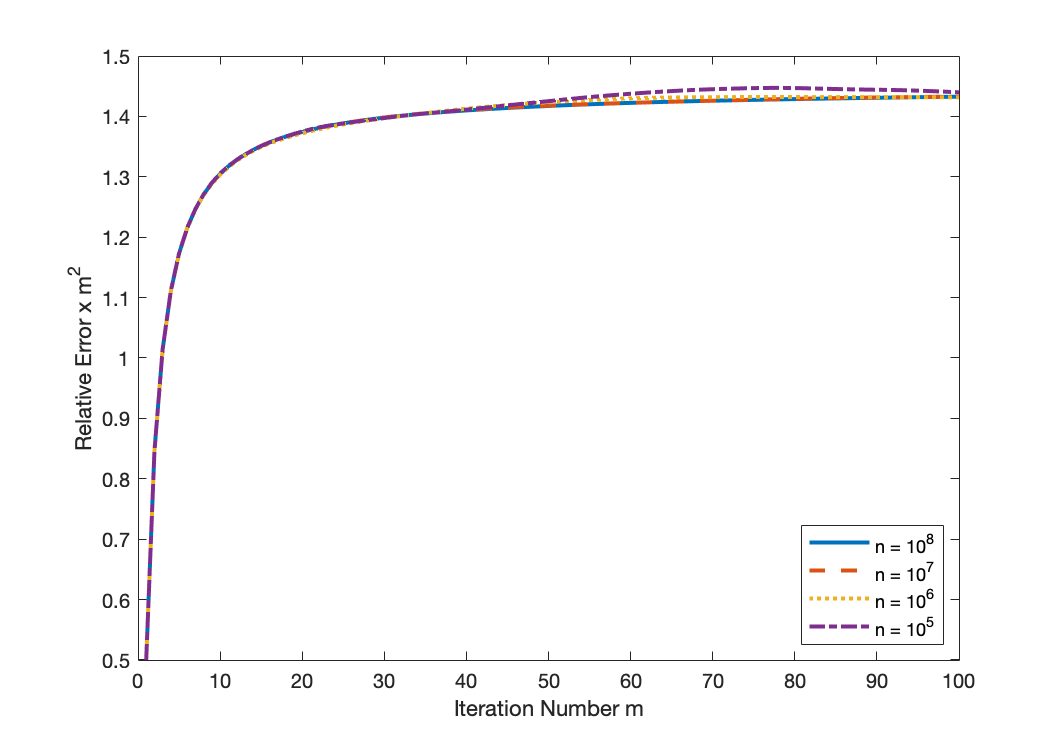}} \subfloat[Box Plot, $\Lambda_{unif}$ ]{\includegraphics*[width=2.5 in, height = 1.7 in]{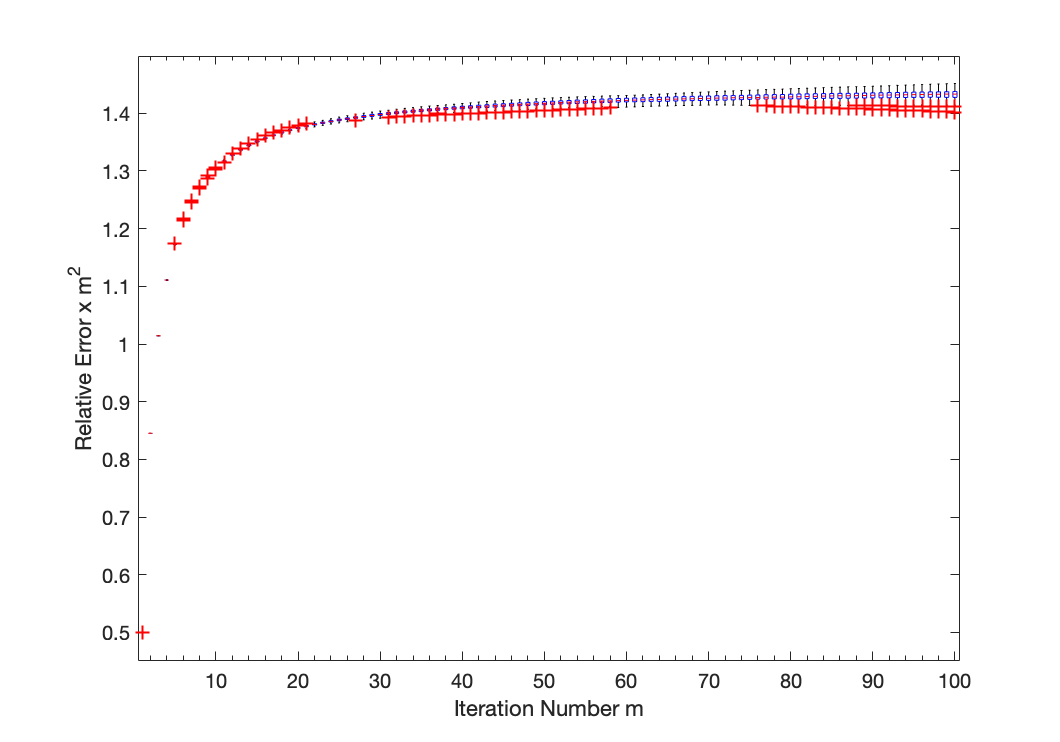}} \\
		\subfloat[Plot, $\Lambda_{semi}$]{\includegraphics*[width=2.5 in, height = 1.7 in]{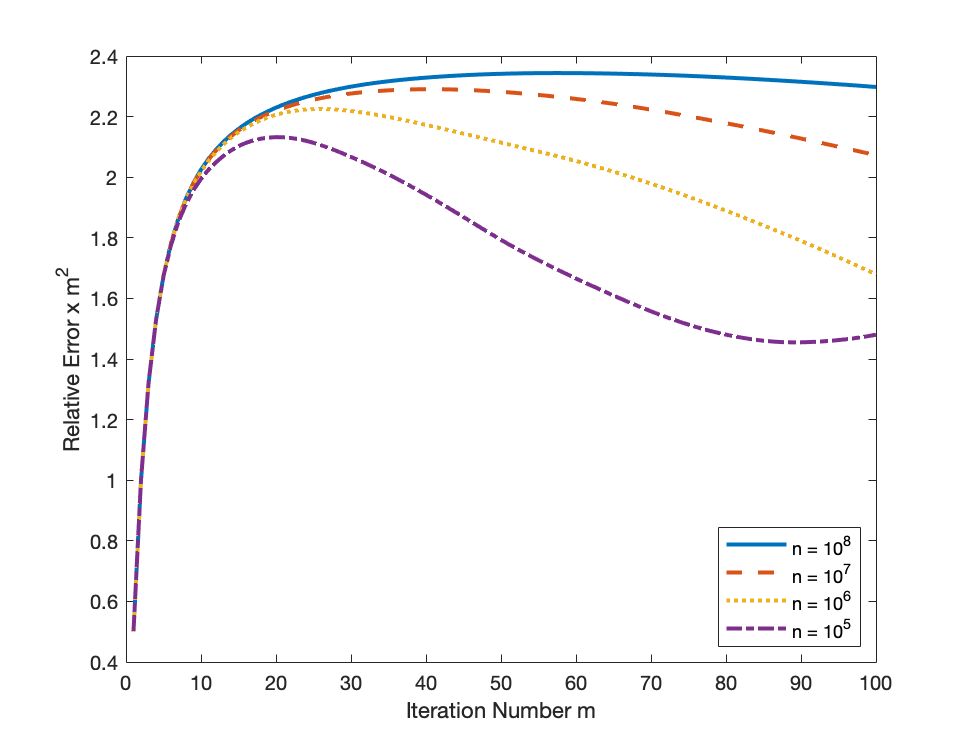}}  \subfloat[Box Plot, $\Lambda_{semi}$]{\includegraphics*[width= 2.5 in, height = 1.7 in]{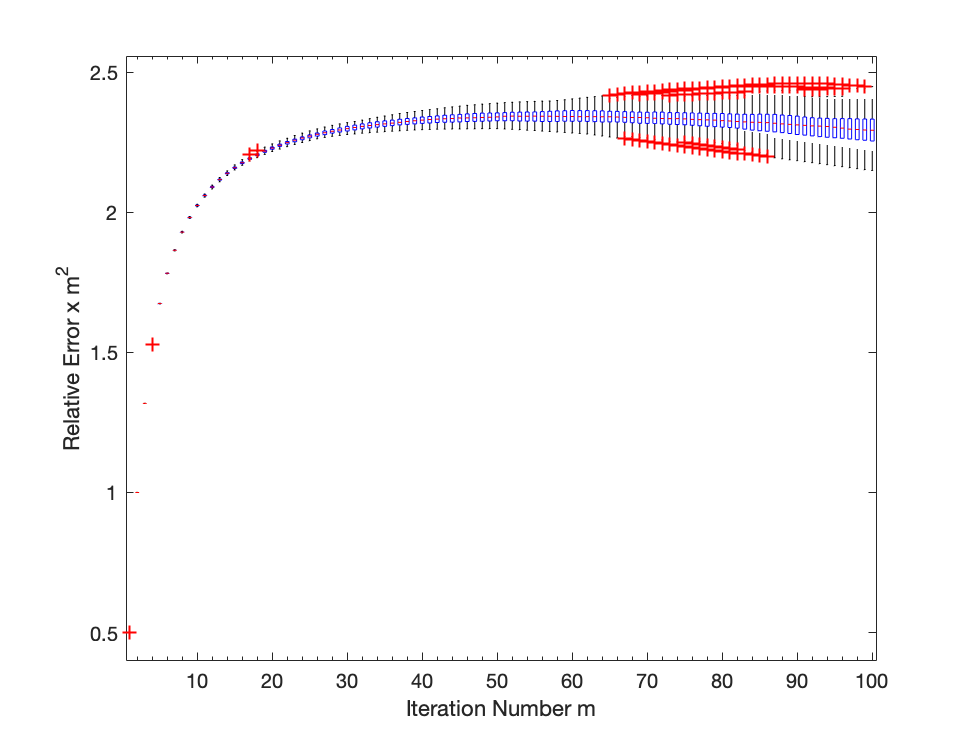}} \\	
		\subfloat[Plot, $\Lambda_{log}$]{\includegraphics*[width=2.5 in, height = 1.7 in]{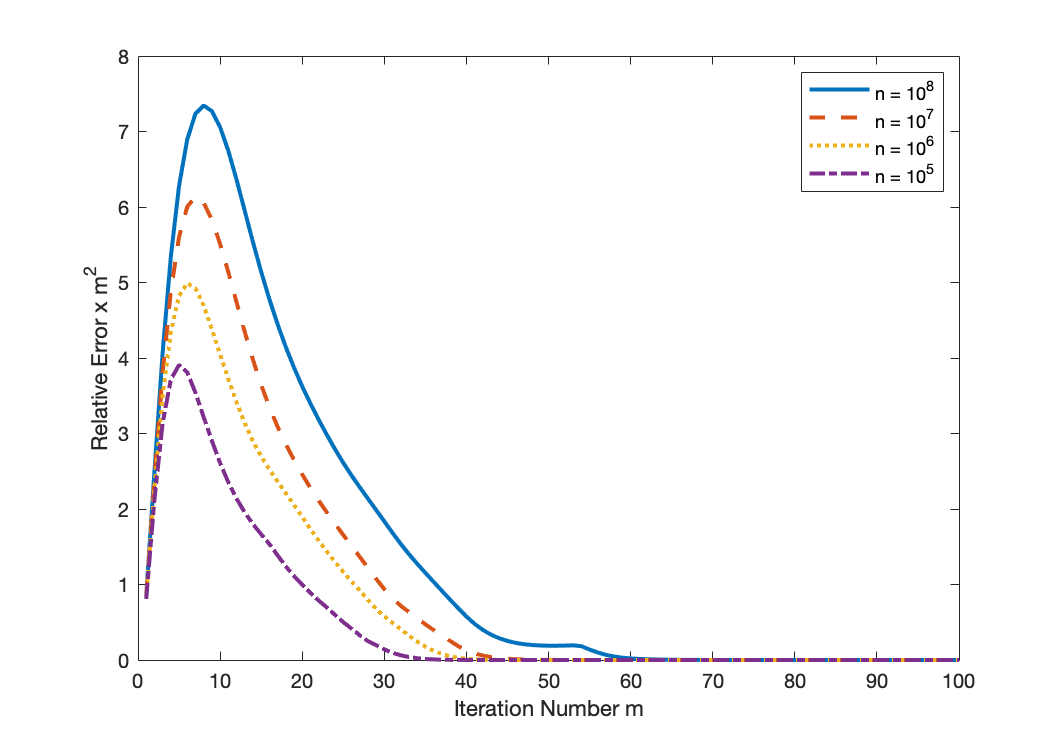}} \subfloat[Box Plot, $\Lambda_{log}$]{\includegraphics*[width= 2.5 in, height = 1.7 in]{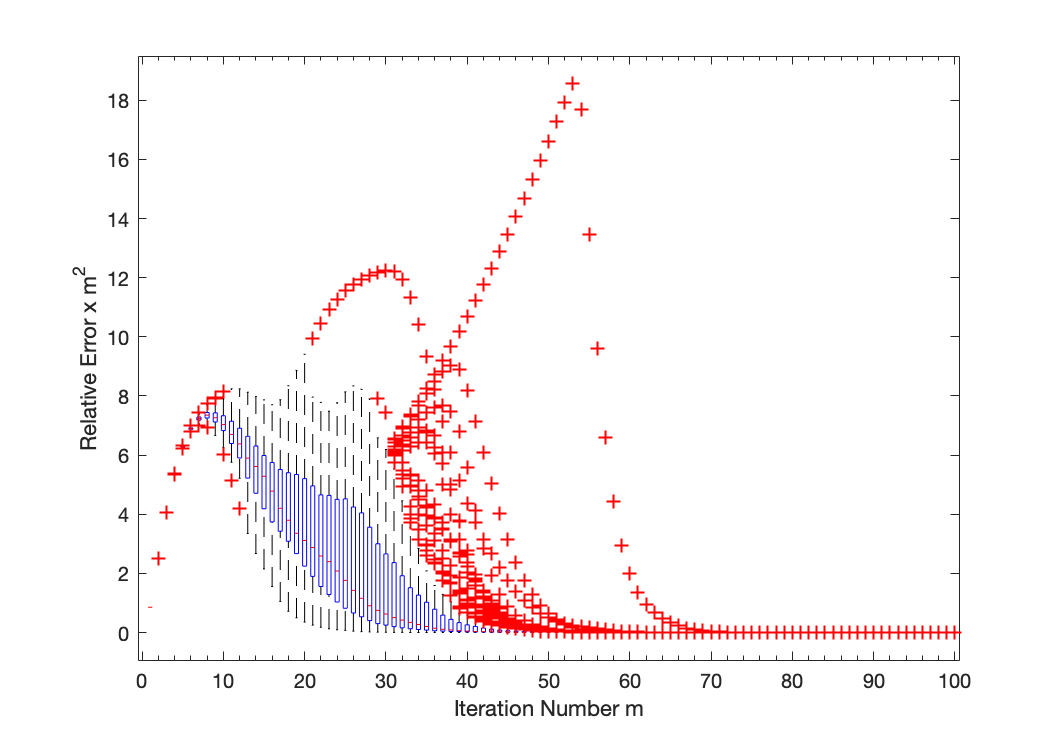}}\\
		
		\caption{Plot and box plot of relative error times $m^2$ vs iteration number $m$ for $\Lambda_{lap}$, $\Lambda_{unif}$, $\Lambda_{semi}$, and $\Lambda_{log}$. The plot contains curves for each dimension $n$ tested. Each curve represents the empirical average relative error for each value of $m$, averaged over $100$ random initializations. The box plot illustrates the variability of relative error for $n = 10^8$. For a given $m$, the $25^{th}$ and $75^{th}$ percentile of the values, denoted by $q_1$ and $q_3$, are the bottom and top of the corresponding box, and the red line in the box is the median. The whiskers extend to the most extreme points in the interval $[q_1 - 1.5(q_3 -q_1), q_3 + 1.5(q_3 - q_1)]$, and outliers not in this interval correspond to the '+' symbol.}
		\label{fig:reg}
	\end{center}
\end{figure}

For each one of these spectra, we perform tests for dimensions $n = 10^i$, $i = 5,6,7,8$. For each dimension, we generate $100$ random vectors $b\sim \mathcal{U}(S^{n-1})$, and perform $m = 100$ iterations of the Lanczos method on each vector. In Figure \ref{fig:reg}, we report the results of these experiments. In particular, for each spectrum, we plot the empirical average relative error $(\lambda_1-\lambda_1^{(m)} )/ (\lambda_1-\lambda_n)$ for each dimension as $m$ varies from $1$ to $100$. In addition, for $n = 10^8$, we present a box plot for each spectrum, illustrating the variability in relative error that each spectrum possesses.

In Figure \ref{fig:reg}, the plots of $\Lambda_{lap}$, $\Lambda_{unif}$, $\Lambda_{semi}$ all illustrate an empirical average error estimate that scales with $m^{-2}$ and has no dependence on dimension $n$ (for $n$ sufficiently large). This is consistent with the theoretical results of the paper, most notably Theorem \ref{thm:conv}, as all of these spectra exhibit suitable convergence of their empirical spectral distribution, and the related integral minimization problems all have solutions of order $m^{-2}$. In addition, the box plots corresponding to $n=10^8$ illustrate that the relative error for a given iteration number has a relatively small variance. For instance, all extreme values remain within a range of length less than $.01 \, m^{-2}$ for $\Lambda_{lap}$, $.1\, m^{-2}$ for $\Lambda_{unif}$, and $.4 \, m^{-2}$ for $\Lambda_{semi}$. This is also consistent with the convergence of Theorem \ref{thm:conv}. The empirical spectral distribution of all three spectra converge to shifted and scaled versions of Jacobi weight functions, namely, $\Lambda_{lap}$ corresponds to $\omega^{-1/2,-1/2}(x)$, $\Lambda_{unif}$ to $\omega^{0,0}(x)$, and $\Lambda_{semi}$ to $\omega^{1/2,1/2}(x)$. The limiting value of $m^2 (1-\xi(m))/2$ for each of these three cases is given by $j_{1,\alpha}^2/4$, where $j_{1,\alpha}$ is the first positive zero of the Bessel function $J_\alpha(x)$, namely, the scaled error converges to $\pi^2/16 \approx .617$ for $\Lambda_{lap}$, $\approx 1.45$ for $\Lambda_{unif}$, and $\pi^2/4 \approx 2.47$ for $\Lambda_{semi}$. Two additional properties suggested by Figure \ref{fig:reg} are that the variance and the dimension $n$ required to observe asymptotic behavior both appear to increase with $\alpha$. This is consistent with the theoretical results, as Lemma \ref{lm:lowerbound_mix} (and $\Lambda_{log}$) results from considering $\omega^{\alpha,0}(x)$ with $\alpha$ as a function of $n$.

The plot of relative error for $\Lambda_{log}$ illustrates that relative error does indeed depend on $n$ in a tangible way, and is increasing with $n$ in what appears to be a logarithmic fashion. For instance, when looking at the average relative error scaled by $m^2$, the maximum over all iteration numbers $m$ appears to increase somewhat logarithmically ($\approx 4$ for $n = 10^5$, $\approx 5$ for $n = 10^6$, $\approx 6$ for $n = 10^7$, and $\approx 7$ for $n = 10^8$). In addition, the boxplot for $n=10^8$ illustrates that this spectrum exhibits a large degree of variability and is susceptible to extreme outliers. These numerical results support the theoretical lower bounds of Section 3, and illustrate that the asymptotic theoretical lower bounds which depend on $n$ do occur in practical computations.

\section*{Acknowledgments}

The author would like to thank Alan Edelman, Michel Goemans and Steven Johnson for interesting conversations on the subject, and Louisa Thomas for improving the style of presentation. The author was supported in part by ONR Research Contract N00014-17-1-2177.

{ \small 
	\bibliographystyle{plain}
	\bibliography{vector} }

\end{document}